\documentclass{article}
\usepackage[english]{babel}
\usepackage{amsmath}
\usepackage{amsthm}
\usepackage{amssymb}

\usepackage{amscd}
\usepackage[latin1]{inputenc}
\usepackage[all]{xy}
\usepackage{MnSymbol}

\newtheorem{teor}{Theorem}[section]
\newtheorem*{teorA}{Theorem A}
\newtheorem*{teorB}{Theorem B}
\newtheorem{defi}{Definition}
\newtheorem{lemma}[teor]{Lemma}
\newtheorem{notation}[teor]{Notation}
\newtheorem{prop}[teor]{Proposition}
\newtheorem{cor}[teor]{Corollary}
\newtheorem{rem}[teor]{Remark}

\newtheorem{exem}[teor]{Example}
\newtheorem{exems}[teor]{Examples}

\topskip=\baselineskip  \headsep=0.3in
\newcommand{\Spec}{\text{Spec}}
\newcommand{\Supp}{\text{Supp}}
\newcommand{\Gen}{\text{Gen}}
\newcommand{\T}{\mathcal{T}}
\newcommand{\D}{\mathcal{D}}
\newcommand{\C}{\mathcal{C}}
\newcommand{\F}{\mathcal{F}}
\newcommand{\p}{\mathbf{p}}
\newcommand{\Mod}{\text{-Mod}}
\newcommand{\Hom}{\text{Hom}}
\newcommand{\monic}{\xymatrix{\ar@{^(->}[r] & }}
\newcommand{\epic}{\xymatrix{\ar@{>>}[r] & }}
\newcommand{\iso}{\xymatrix{\ar[r]^{\sim} & }}
\newcommand{\flecha}{\xymatrix{\ar[r] &}}
\newcommand{\U}{\mathcal{U}_{\phi}}
\newcommand{\Ker}{\text{Ker}}
\newcommand{\Coker}{\text{Coker}}
\newcommand{\Imagen}{\text{Im}}
\newcommand{\Ht}{\mathcal{H}_{\phi}}
\newcommand{\limite}{\varinjlim_{\Ht}}
\newcommand{\TF}{\mathcal{TF}}

\def\Hom{\mathop{\rm Hom}\nolimits}

\def\End {\mathop{\rm End}\nolimits}

\def\Ext {\mathop{\rm Ext}\nolimits}

\def\Ker {\mathop{\rm Ker}\nolimits}
\def\Im {\mathop{\rm Im}\nolimits}

\title{Hearts of t-structures in the derived category of a commutative Noetherian ring}

\author{Carlos E. Parra \\
Departamento de Matem\'aticas\\ Universidad de los Andes \\ ({\bf 5101}) M\'erida\\ VENEZUELA\\
{\it carlosparra@ula.ve} \\  \\ Manuel Saor\'in  \\ Departamento de Matem\'aticas\\
Universidad de Murcia, Aptdo. 4021\\
30100 Espinardo, Murcia\\
SPAIN\\ {\it msaorinc@um.es} }

\begin{document}
\date{}
\maketitle

\footnote{Parra is supported by a grant from the Universidad de los Andes (Venezuela) and Saor\'in is
supported by research projects from the Spanish Ministry of
Education (MTM2010-20940-C02-02) and from the Fundaci\'on 'S\'eneca'
of Murcia (04555/GERM/06), with a part of FEDER funds. The authors
thank these institutions for their help. We also thank Leovigildo Alonso and Ana Jerem\'ias for their remarks and for calling our attention on Ana's thesis. Finally, we thank the referee for the careful reading of the manuscript and his/her suggestions and comments.}

\begin{abstract}
{\bf Let $R$ be a commutative Noetherian ring and let $\mathcal D(R)$ be its (unbounded) derived category.    We show that all  compactly generated t-structures in $\mathcal D(R)$ associated to a left bounded filtration by supports of $\Spec (R)$ have a heart which is a Grothendieck category. Moreover, we identify all compactly generated t-structures  in $\mathcal D(R)$ whose heart is a module category. As geometric consequences for a compactly generated t-structure  $(\mathcal{U},\mathcal{U}^\perp [1])$ in the derived category $\mathcal{D}(\mathbb{X})$ of an affine Noetherian scheme $\mathbb{X}$, we get the following: 1) If the sequence  $(\mathcal{U}[-n]\cap\mathcal{D}^{\leq 0}(\mathbb{X}))_{n\in\mathbb{N}}$ is stationary, then the heart $\mathcal{H}$ is a Grothendieck category; 2) If  $\mathcal{H}$ is a module category, then $\mathcal{H}$ is always equivalent to $\text{Qcoh}(\mathbb{Y})$, for some affine subscheme $\mathbb{Y}\subseteq\mathbb{X}$; 3) If $\mathbb{X}$ is connected, then: a)  when $\bigcap_{k\in\mathbb{Z}}\mathcal{U}[k]=0$, the heart $\mathcal{H}$ is a module category if, and only if, the given t-structure is a translation of the canonical t-structure in $\mathcal{D}(\mathbb{X})$; b) when $\mathbb{X}$ is irreducible, the heart $\mathcal{H}$ is a module category if, and only if, there are an affine subscheme $\mathbb{Y}\subseteq\mathbb{X}$ and an integer $m$ such that $\mathcal{U}$ consists of the complexes $U\in\mathcal{D}(\mathbb{X})$ such that the support of $H^j(U)$ is in $\mathbb{X}\setminus\mathbb{Y}$, for all $j>m$. }
\end{abstract}

{\bf Mathematics Subjects Classification:} 18E30, 13Dxx, 14xx, 16Exx

\section{Introduction}
T-structures in triangulated categories were introduced by Beilinson, Bernstein and Deligne \cite{BBD} in their study of perverse
sheaves on an algebraic or analytic variety. A t-structure in a triangulated category $\mathcal{D}$  is a pair of full
subcategories satisfying suitable axioms (see the precise definition
in next section) which guarantee that their intersection is an
abelian category $\mathcal{H}$, called the heart of the t-structure. One then naturally defines a cohomological functor $\tilde{H}:\mathcal{D}\longrightarrow\mathcal{H}$, a fact which allows to develop an intrinsic (co)homology theory,
where the homology 'spaces' are again objects of $\mathcal{D}$
itself.

T-structures have been used in many
branches of Mathematics, with special impact in Algebraic Geometry
and Representation Theory of Algebras. One line of research in the
topic has been the explicit construction, for concrete triangulated
categories, of wide classes of t-structures. This approach has led
to classification results in many cases  (see, e.g., \cite{Br},
\cite{GKR}, \cite{AJS}, \cite{St}, \cite{ST}, \cite{KN}...). A second line of
research consists in starting with a well-behaved class of
t-structures and try to find necessary and sufficient conditions on
a t-structure in the class so that the heart is a `nice' abelian
category. For instance, to give conditions for the heart to be a Grothendieck or even a
module category. So far, in this second line of research, the focus has been almosts exclusively put on the class of the so-called t-structures of Happel-Reiten-Smal\o \  \cite{HRS}.  Given a torsion pair $\mathbf{t}$ in an abelian category $\mathcal{A}$, these authors construct a t-structure in the bounded derived category $\mathcal{D}^b(\mathcal{A})$, which is the restriction of a t-structure in $\mathcal{D}(\mathcal{A})$ whenever this last category is well-defined (i.e. the Hom groups are sets). Note that even $\mathcal{D}^b(\mathcal{A})$ might have Hom groups which are not sets (see \cite{CN}).   The study of conditions under which the heart of this t-structure is a Grothendieck or module category has deserved a lot of attention in recent times (see \cite{HKM}, \cite{CGM}, \cite{CMT}, \cite{CG}, \cite{MT}, \cite{PS} and \cite{PS2}). 

In this  paper we quit the Happel-Reiten-Smal\o \ setting and move to study these questions for the heart of a compactly generated t-structure in the derived category $\mathcal D(R)$ of a commutative Noetherian ring $R$. Of course, if $\mathbb{X}=\Spec(R)$ is the associated affine scheme, then we have an equivalence of categories $\text{Qcoh}(\mathbb{X})\cong R\Mod$, where $\text{Qcoh}(\mathbb{X})$ is the category of quasi-coherent sheaves on $\mathbb{X}$. Hence, we  also have $\mathcal D(\mathbb{X}):=\mathcal D(\text{Qcoh}(\mathbb{X}))\cong\mathcal D(R)$. Since all Noetherian affine schemes are of this form (see \cite[Chapter 2]{H}), our results have a geometric interpretation. Roughly speaking, our results show that the hearts of essentially all compactly generated t-structures in $\mathcal{D}(\mathbb{X})$ are Grothendieck categories, while, when $\mathbb{X}$ is connected,  the ones whose heart is a module category (in principle over a not necessarily commutative ring) are obtained from translations of the canonical t-structure on $\mathcal{D}(\mathbb{Y})$, for some affine subscheme $\mathbb{Y}\subseteq\mathbb{X}$. Unlike the case when $\mathbb{X}$ is projective (see \cite{B}), this latter fact discards the appearance of module categories over noncommutative rings as hearts of t-structures in the derived category of an affine Noetherian scheme.

Concretely, the  following are geometric consequences of the two main results of the paper (see  Theorem \ref{teo. H is a Grothendieck c.} and Corollaries \ref{cor.modular case for left non-degenerate} and \ref{cor.modular for prime nilradical}). In the second result and all throughout the paper,  'module category' means a category which is equivalent to the category of (all) modules over an associative ring with unit, not necessarily commutative (see Subsection \ref{sub. Grothe. and module}).

\begin{teorA} \label{teor.Theorem A}
Let $\mathbb{X}$ be an affine Noetherian scheme and let $(\mathcal{U},\mathcal{U}^\perp [1])$ be a compactly generated t-structure in $\mathcal{D}(\mathbb{X})$ such that $\mathcal{U}\neq\mathcal{U}[-1]$ and the chain of subcategories $(\mathcal{U}[-n]\cap\mathcal{D}^{\leq 0}(\mathbb{X}))_{n\in\mathbb{N}}$ is stationary. The heart of this t-structure is a Grothendieck category. 
\end{teorA}

\begin{teorB} \label{teore.Theorem B}
 Let $\mathbb{X}$ be an affine Noetherian scheme, let $(\mathcal{U},\mathcal{U}^\perp [1])$ be a compactly generated t-structure in the derived category  $\mathcal D(\mathbb{X}):=\mathcal{D}(\text{Qcoh}\mathbb{X})$ such that $\mathcal{U}[-1]\neq\mathcal{U}\neq 0$ and let $\mathcal{H}$ be the heart of this t-structure. The following assertions hold:

\begin{enumerate}
\item If $\mathcal{H}$ is a module category, then there exists an affine subscheme $\mathbb{Y}$ of $\mathbb{X}$ such that $\mathcal{H}\cong\text{Qcoh}(\mathbb{Y})$. 

\item When $\mathbb{X}$ is connected, the following statements hold true:

\begin{enumerate}
\item Suppose that  $(\mathcal{U},\mathcal{U}^\perp [1])$ is left  nondegenerate (i.e. $\bigcap_{k\in\mathbb{Z}}\mathcal{U}[k]=0$). Then $\mathcal H$ is a module category if, and only if, there exists an integer $m$ such that  $(\mathcal{U},\mathcal{U}^\perp [1])=(\mathcal{D}^{\leq m}(\mathbb{X}),\mathcal{D}^{\geq m}(\mathbb{X}))$. 
\item Suppose that $\mathbb{X}$ is irreducible. Then $\mathcal{H}$ is a module category if, and only if, there exist an integer $m$ and a nonempty  affine subscheme $\mathbb{Y}\subseteq\mathbb{X}$ such that 
  $\mathcal{U}$ consists of the complexes $U\in\mathcal D(\mathbb{X})$ such that $\Supp (H^i(U))\subseteq\mathbb{X}\setminus\mathbb{Y}$, for all $i>m$. In this case $\mathcal{H}$ is equivalent to $\text{Qcoh}(\mathbb{Y})$. 
\end{enumerate}
\end{enumerate}
\end{teorB}

The organization of the paper goes as follows. The reader is referred to Section 2 for the pertinent definitions. In Section 2 we introduce and recall the concepts and results which are most relevant for the latter sections of the paper. In particular, we recall the bijection (see \cite[Theorem 3.1]{AJS}) between compactly generated t-structures of $\mathcal{D}(R)$ and filtrations by supports of $\Spec (R)$.  In Section 3 we give some properties of the heart $\mathcal{H}$ of a compactly generated t-structure in the derived category $\mathcal{D}(R)$ of a commutative Noetherian ring, showing in particular  that $\mathcal{H}$ always has a generator (Proposition \ref{generator of H}) and that the AB5 condition on $\mathcal{H}$ is a local property (Corollary \ref{cor. Hp is AB5}). In Section 4 we give the first main result of the paper, corresponding to Theorem A above, which states that if $\phi$ is a left bounded filtration by supports of $\Spec(R)$, then the heart of the associated t-structure is a Grothendieck category (see Theorem \ref{teo. H is a Grothendieck c.}). The proof is done by settling first the case when the filtration is eventually trivial (Theorem \ref{teo. AB5 finite}). The final Section 5 deals with the case in which the heart is a module category and it includes the second main result of the paper (see Theorem \ref{teor.modular main theorem1}). This result determines all the compactly generated t-structures in $\mathcal{D}(R)$ whose heart is a module category and has the above Theorem B as a corollary. The hardest-to-prove part of the theorem says, in essence, that if the heart of a compactly generated t-structure $(\mathcal{U},\mathcal{U}^\perp [1])$ in $\mathcal{D}(R)$ is a module category, then there is a stable under specialization subset $Z\subset\Spec (R)$ such that if $R_Z$ is the associated ring of fractions, then there is a decomposition of $1$ in $R_Z$ as a sum of orthogonal idempotents, $1=\sum_{0\leq k\leq t}e_k$, such that: a) if an $R_Z$-module $N$ has its support in $Z$ when viewed as an $R$-module, then $e_kN=0$, for $k\neq 0$; b) a complex $U\in\mathcal D(R)$ is in $\mathcal{U}$ if, and only if, $\Supp (e_0H^i(R_Z\otimes_R^{\mathbf{L}}U))\subseteq Z$, for all $i\in\mathbb{Z}$, and $e_k(R_Z\otimes_R^\mathbf{L}U)\in\mathcal D^{\leq m_k}(R_Ze_k)$, for $k=1,\dots,t$ and certain integers $m_1<m_2<\dots<m_t$.

\section{Terminology and preliminaries} \label{sec.Terminology}
All rings in this paper are associative with identity. 

\subsection{Some points  of Commutative Algebra} 

In this paragraph we assume that $R$ is commutative Noetherian ring and we denote by $\Spec (R)$ its spectrum.  For each ideal $\mathbf{a} \subseteq R$ we  put $\text{V}(\mathbf{a}):=\{\mathbf{p}\in \Spec(R)| \mathbf{a} \subseteq \mathbf{p}\}$. These subsets are the closed subsets for the \emph{Zariski topology} in $\Spec (R)$. The so defined topological space is noetherian, i.e. satisfies ACC on open subsets. When one writes $1$ as a sum $1=\sum_{1\leq i\leq n}e_i$ of primitive orthogonal idempotents, we have a ring decomposition $R\cong Re_1\times \dots \times Re_n$ and can identify $\Spec(Re_i)$ with $V(R(1-e_i))$. These are precisely the connected components of $\Spec (R)$ as a topological space, so that we have a disjoint union $\Spec (R)=\bigcup_{1\leq i\leq n}\Spec(Re_i)$. We  call the rings $Re_i$ or the idempotents $e_i$ the \emph{connected components of $R$}. 

A subset $Z \subseteq \Spec(R)$ is stable under specialization if, for any couple of prime ideals $\mathbf{p} \subseteq \mathbf{q}$ with $\mathbf{p}\in Z$, it holds that $\mathbf{q}\in Z$. Equivalently, when $Z$ is a union of closed subsets of $\Spec(R)$. Such a subset will be called \emph{sp-subset} in the sequel. The typical example is the \emph{support of an $R$-module} $N$, denoted $\Supp (N)$,  which consists of the prime ideals $\mathbf{p}$ such that $N_\mathbf{p}\cong R_{\mathbf{p}} \otimes_{R} N\neq 0$. Here $(?)_\mathbf{p}$ denotes the localization at $\mathbf{p}$, both when applied to $R$-modules or to the ring $R$ itself.  Recall that the assignment $N\rightsquigarrow N_\mathbf{p}$ defines an exact functor $R\Mod\longrightarrow R_\mathbf{p}\Mod$. If $f:R\longrightarrow R_\mathbf{p}$ is the canonical ring homomorphism, then we have a natural isomorphism $(?)_\mathbf{p}\cong f^*$, where $f^*$ is the extension of scalars with respect to $f$. In particular, we have an adjoint pair of exact functors $(f^*,f_*)$, where $f_*:R_\mathbf{p}\Mod\longrightarrow R\Mod$ is the restriction of scalars. 

Recall that if $M$ is an $R$-module, then  the set of  \emph{associated primes} of $M$, denoted $\text{Ass}(M)$, is the subset of $\Supp (M)$ consisting of those $\mathbf{p}\in\Spec (R)$ such that $\mathbf{p}=\{a\in R:$ $ax=0\}$, for some $x\in M\setminus\{0\}$. 

\subsection{Grothendieck and module categories}\label{sub. Grothe. and module}

Given any category $\mathcal{C}$, a \emph{generator} in $\mathcal{C}$ is an object $G$ such that the functor $\Hom_\mathcal{C}(G,?):\mathcal C\longrightarrow\text{Sets}$ is faithful. In the case when $\mathcal{C}=\mathcal{A}$ is an abelian category with coproducts, an object $G$ is a generator precisely when all objects of $\mathcal A$ are epimorphic image of coproducts of copies of $G$. In this case, a \emph{set of generators} of $\mathcal{A}$ will be a set $\mathcal{S}$ of objects  such that $\coprod_{S\in\mathcal{S}}S$ is a generator of $\mathcal{A}$. Note that $\mathcal{A}$ has a generator if, and only if, it has a set of generators. 

In general, given any additive category $\mathcal{C}$ with coproducts, an object $X$ will be called \emph{compact}, when the functor $\Hom_{\mathcal C}(X,?):\mathcal C\longrightarrow Ab$ preserves coproducts.
Let $I$ be a directed set and  $[(X_i)_{i\in
I},(u_{ij})_{i\leq j}]$ be an $I$-direct system in $\mathcal C$. We put $X_{ij}=X_i$, for all $i\leq j$. The
\emph{colimit-defining morphism} associated to the direct system is
the unique morphism $f:\coprod_{i\leq
j}X_{ij}\longrightarrow\coprod_{i\in I}X_i$ such that if
$\lambda_{kl}:X_{kl}\longrightarrow\coprod_{i\leq j}X_{ij}$ and
$\lambda_j:X_j\longrightarrow\coprod_{i\in I}X_i$ are the canonical
morphisms into the coproducts, then
$f\circ\lambda_{ij}=\lambda_i-\lambda_j\circ u_{ij}$ for all $i\leq
j$. It is well-known (see, e.g. \cite[Proposition II.6.2]{P}) that the direct system has a colimit (i.e. $\varinjlim X_i$ exists) if, and only if, $f$ has a cokernel, in which case we have $\varinjlim X_i=\Coker(f)$.  

Recall that a cocomplete  abelian category is AB5 when direct limits are exact in the category. If, in addition, the category has a generator, we call it a \emph{Grothendieck category}. The main example of Grothendieck category is the category $A\Mod$ of (left) modules over a not necessarily commutative (but always associate unital in this paper) ring $A$. Any category equivalent to $A\Mod$ for such a ring will be called a \emph{module category}. Recall that an object $P$ of a cocomplete abelian category is a \emph{progenerator} when it is a compact projective generator.  The following well-known result of Gabriel and Mitchell (see \cite[Corollary 3.6.4]{Po}) identifies the module categories within the class of cocomplete abelian categories.

\begin{prop} \label{prop.Gabriel-Mitchell}
Let $\mathcal{A}$ be a cocomplete abelian category. It is a module category if, and only if, it has a progenerator. In such case, if $P$ is a progenerator of $\mathcal{A}$, then the functor $\Hom_\mathcal{A}(P,?):\mathcal{A}\longrightarrow\End_\mathcal{A}(P)^{op}\Mod$ is an equivalence of categories.  
\end{prop}

\subsection{Hereditary torsion pairs}
Let $\mathcal{G}$ be a Grothendieck category.  A \emph{hereditary torsion class} or \emph{localizing subcategory} of $\mathcal{G}$ is a full subcategory $\mathcal{T}$ which is closed under taking subobjects, quotients, extensions and arbitrary coproducts. In such case, if $\mathcal F=\mathcal T^\perp$ is the class of objects $F$ such that $\Hom_R(T,F)=0$, for all $T\in\mathcal T$, then  $\mathbf{t}=(\mathcal T,\mathcal F)$ is a \emph{hereditary torsion pair} or \emph{hereditary torsion theory} in $\mathcal{G}$ (see \cite[Chapter VI]{S}). Note that the class $\mathcal{F}$ is closed under taking injective envelopes in $\mathcal G$. When $\mathcal G= A\Mod$ is the module category over a not necessarily commutative ring $A$, a module  in $\mathcal T$ is called a \emph{torsion module} while a module in $\mathcal F$ is called a \emph{torsionfree module}. 

If, for each  object $M$ of $\mathcal G$, we denote by $t(M)$ the largest subobject of $M$ which is $\mathcal T$, then the assignments $M\rightsquigarrow t(M)$ and $M\rightsquigarrow (1:t)(M):=M/t(M)$ define idempotent functors $t, (1:t):\mathcal{G}\longrightarrow\mathcal{G}$. We will call $t$ (resp. $(1:t)$) the \emph{torsion radical (resp. coradical)} associated to the torsion pair.

When $A=R$ is commutative Noetherian and $\mathcal G=R\Mod$, we have the following precise description (see \cite[Section 6.4]{G} \cite[Subsection VI.6.6]{S}):

\begin{prop} \label{prop.sp-subsets versus torsion pairs}
Let $R$ be a commutative Noetherian ring. The assignment $Z\rightsquigarrow (\mathcal T_Z,\mathcal T_Z^\perp)$ defines a bijection between the sp-subsets of $\Spec (R)$ and the hereditary torsion pairs in $R\Mod$, where $\mathcal T_Z$ is the class of modules $T$ such that $\Supp (T)\subseteq Z$. Its inverse takes $(\mathcal T,\mathcal T^\perp)$ to the set $Z_{\mathcal T}$ of prime ideals $\mathbf{p}$ such that $R/\mathbf{p}$ is in $\mathcal{T}$. Moreover $\mathcal T_Z^\perp$ consists of the $R$-modules $F$ such that $\text{Ass}(F)\cap Z=\emptyset$. 
\end{prop}

In the situation of last proposition, we will denote by $\Gamma_Z$ the torsion radical associated to the torsion pair $(\mathcal T_Z,\mathcal T_Z^\perp)$. It is worth noting that here, not only $\mathcal T_Z^\perp$, but also $\mathcal T_Z$ is closed under taking injective envelopes in $R\Mod$ (see \cite[Corollaire V.6.3]{G})

The following will be frequently used in the paper:

\begin{rem}\label{remark class closed direct.limit}
If $Z\subseteq\Spec (R)$ is an sp-subset, then the classes $\T_{Z}$ and $\F_{Z}=\mathcal T_Z^\perp$ are closed under taking direct limits in $R\Mod$. As a consequence,    the associated  torsion radical (resp.
coradical) functor $\Gamma_Z:R\Mod \flecha R\Mod$ (resp.
$(1:\Gamma_Z):R\Mod \flecha R\Mod$) preserves direct limits.
\end{rem}

\subsection{T-structures in triangulated categories with coproducts}\label{subsection of $L$}

Let $(\mathcal{D},?[1])$ be a triangulated category. A
\emph{t-structure} in $\mathcal{D}$ is a pair
$(\mathcal{U},\mathcal{W})$ of full subcategories, closed under
taking direct summands in $\mathcal{D}$,  which satisfy the
 following  properties:

\begin{enumerate}
\item[i)] $\text{Hom}_\mathcal{D}(U,W[-1])=0$, for all
$U\in\mathcal{U}$ and $W\in\mathcal{W}$;
\item[ii)] $\mathcal{U}[1]\subseteq\mathcal{U}$;
\item[iii)] For each $X\in Ob(\mathcal{D})$, there is a triangle $U\longrightarrow X\longrightarrow
V\stackrel{+}{\longrightarrow}$ in $\mathcal{D}$, where
$U\in\mathcal{U}$ and $V\in\mathcal{W}[-1]$.
\end{enumerate}
It is easy to see that in such case $\mathcal{W}=\mathcal{U}^\perp
[1]$ and $\mathcal{U}={}^\perp (\mathcal{W}[-1])={}^\perp
(\mathcal{U}^\perp )$. For this reason, we will write a t-structure
as $(\mathcal{U},\mathcal{U}^\perp [1])$. We will call $\mathcal{U}$
and $\mathcal{U}^\perp$ the \emph{aisle} and the \emph{co-aisle} of
the t-structure, respectively. The objects $U$ and $V$ in the above
triangle are uniquely determined by $X$, up to isomorphism, and
define functors
$\tau_\mathcal{U}:\mathcal{D}\longrightarrow\mathcal{U}$ and
$\tau^{\mathcal{U}^\perp}:\mathcal{D}\longrightarrow\mathcal{U}^\perp$
which are right and left adjoints to the respective inclusion
functors. We call them the \emph{left and right truncation functors}
with respect to the given t-structure. Note that
$(\mathcal{U}[k],\mathcal{U}^\perp [k])$ is also a t-structure in
$\mathcal{D}$, for each $k\in\mathbb{Z}$. The full subcategory
$\mathcal{H}=\mathcal{U}\cap\mathcal{W}=\mathcal{U}\cap\mathcal{U}^\perp
[1]$ is called the \emph{heart} of the t-structure and it is an
abelian category, where the short exact sequences 'are' the
triangles in $\mathcal{D}$ with their three terms in $\mathcal{H}$.
Moreover, with the obvious abuse of notation,  the assignments
$X\rightsquigarrow (\tau_{\mathcal{U}}\circ\tau^{\mathcal{U}^\perp
[1]})(X)$ and $X\rightarrow (\tau^{\mathcal{U}^\perp
[1]}\circ\tau_\mathcal{U})(X)$ give  naturally isomorphic functors $\mathcal D\longrightarrow\mathcal{H}$ and, hence, they  define
a functor $\tilde{H}:\mathcal{D}\longrightarrow\mathcal{H}$. This functor is
cohomological (see \cite{BBD}). It is shown in \cite[Lemma 3.1]{PS}  that its restriction $L:=\tilde{H}_{|\mathcal{U}}:\mathcal{U}\longrightarrow\mathcal H$ is left adjoint to the inclusion $j:\mathcal{H} \hookrightarrow \mathcal{U}$.

When $\mathcal{D}$ is a triangulated category with  coproducts,  a
 t-structure $(\mathcal{U},\mathcal{U}^\perp [1])$ in $\mathcal D$ is called
\emph{compactly generated} when there is a set
$\mathcal{S}\subset\mathcal{U}$, consisting of compact objects,
such that $\mathcal{U}^\perp$ consists of the $Y\in\mathcal{D}$ such
that $\text{Hom}_\mathcal{D}(S[n],Y)=0$, for all $S\in\mathcal{S}$
and integers $n\geq 0$. In such case, we say that $\mathcal{S}$ is a
\emph{set of compact generators of the aisle $\mathcal{U}$}. A t-structure $(\mathcal{U},\mathcal{U}^{\perp}[1])$ in $\D$ will be called \emph{left nondegenerate}, when $\underset{n\in \mathbb{Z}}{\bigcap} \mathcal{U}[n]=0$.

\subsection{Compactly generated t-structures in $\mathcal D (R)$}\label{sub. compactly generated}
Let $R$ be again a commutative Noetherian ring and let 
 $\mathcal{P}(S)$ denote the set of parts of $S$, for any set $S$. A \emph{filtration by supports} of $\Spec(R)$ is a decreasing map $\xymatrix{\phi:\mathbb{Z} \ar[r] & \mathcal{P}(\Spec(R))}$ such that $\phi(i)\subseteq \Spec(R)$ is a sp-subset for each $i\in \mathbb{Z}$. Filtrations by supports have turned out to be a fundamental tool to classify objects in the derived category of $R$. For instance, when the Krull dimension of $R$ is finite,  Stanley (see \cite[Theorem A]{St}) gave a bijection between the  nullity classes in the  subcategory $\mathcal{D}^b_{\text{fg}}(R)$ of complexes with finitely generated bounded homology and the filtrations by support of $\text{Spec}(R)$. Simultaneously, another bijection,  with 'nullity classes in $\mathcal{D}^b_{\text{fg}}(R)$' replaced by 'compactly generated t-structures in $\mathcal{D}(R)$' and without restriction  on the Krull dimension, was given by Alonso, Jerem\'ias and Saor\'in (see \cite[Theorem 3.1]{AJS}). In this paper, we will frequently use the last mentioned bijection and the terminology of  that paper.

To abbreviate, we will refer to a filtration by supports of $\Spec(R)$ simply by a \emph{sp-filtration} of $\Spec(R)$. For each integer $k$, we denote by $(\mathcal T_k,\mathcal F_k)$ the hereditary torsion pair in $R\Mod$ associated to the sp-subset $\phi (k)$ (see Proposition \ref{prop.sp-subsets versus torsion pairs}), and will put $t_k=\Gamma_{\phi (k)}$ to denote the associated torsion radical. 
An sp-filtration $\phi$ will be called \emph{left  bounded} when there exists an integer $n$ such that $\phi (i)=\phi (n)$, for all $i\leq n$. We will call $\phi$ \emph{eventually trivial} when there is an integer $i$ such that $\phi (i)=\emptyset$. 

If $\mathcal{S}$ is any set of objects in $\mathcal D(R)$ then, by \cite[Proposition 3.2]{AJSo}, we know that the smallest full subcategory of $\D(R)$ closed under extensions, arbitrary coproducts and application of the shift functor $?[1]$ is an aisle of $\mathcal D(R)$. We will call it the \emph{aisle generated by $\mathcal{S}$} and will denote it by $\text{aisle} (\mathcal{S})$.
The corresponding t-structure in $\D(R)$ will be called the \emph{t-structure generated by $\mathcal{S}$}.  The co-aisle $\text{aisle}(\mathcal{S})^\perp$ consists of the complexes $Y$ such that $\Hom_{\D(R)}(S[k],Y)=0$, for all $S\in\mathcal S$ and all integers $k\geq 0$.

Given an  sp-filtration $\phi$ of $\Spec(R)$, we shall denote by $\mathcal U_\phi$ the aisle generated by all the stalk complexes $R/\mathbf{p}[-i]$, such that $i\in\mathbb{Z}$ and $\mathbf{p}\in\phi (i)$. The following is \cite[Theorem 3.11]{AJS}. In its statement $\mathbf{R}\Gamma_Z:\mathcal D(R)\longrightarrow\mathcal D(R)$ denotes the right derived functor of $\Gamma_Z$, for any sp-subset $Z$ of $\Spec (R)$.

\begin{prop} \label{prop.sp-filtrations versus t-structures}
The assignment $\phi\rightsquigarrow (\mathcal U_\phi,\mathcal U_\phi^\perp [1])$ gives a bijection between the filtrations by supports of $\Spec (R)$ and the compactly generated t-structures of $\D(R)$. Its inverse takes $(\mathcal U,\mathcal U^\perp)\rightsquigarrow\phi_\mathcal{U}$, where $\phi_{\mathcal U}(i)$ consists of the prime ideals $\p$ such that $R/\mathbf{p}[-i]\in\mathcal{U}$, for all $i\in\mathbb{Z}$. Moreover, in this situation, we have the following explicit descriptions:
 
\begin{center}                         
$\mathcal{U}_\phi =\{X\in\mathcal D(R):$ $\Supp (H^i(X))\subseteq\phi (i)\text{, for all }i\in\mathbb{Z}\}$

$\mathcal{U}_\phi^\perp =\{Y\in\mathcal{D}(R):$ $\mathbf{R}\Gamma_{\phi (i)}(Y)\in\mathcal D^{>i}(R)\text{, for all }i\in\mathbb{Z}\}$.
\end{center}
\end{prop}

 We will denote by $\mathcal{H}_{\phi}$ the heart of the t-structure $(\mathcal{U}_{\phi},\mathcal{U}_{\phi}^{\perp}[1]).$ The left and right truncation functors associated to $(\mathcal U_\phi ,\mathcal U_\phi^\perp [1])$ will be denoted by $\tau_\phi^{\leq}$ and $\tau_\phi^{>}$, respectively. We warn the reader not to confuse them with the truncation functors $\tau^{\leq m}$ and $\tau^{>m}$ associated to the canonical t-structure $(\mathcal{D}^{\leq m}(R),\mathcal{D}^{\geq m}(R))$, where $m$ is an integer. 

Concerning this last notation, whenever $n\in\mathbb{Z}$, we will denote by $\mathcal D^{\leq n}(R)$ (resp. $\mathcal D^{<n}(R)$), resp. $\mathcal D^{\geq n}(R)$ resp $\mathcal D^{> n}(R)$), the full subcategory of $\mathcal D(R)$ consisting of the complexes $M$ such that $H^i(M)=0$, for all $i>n$ (resp. $i\geq n$, resp.  $i<n$, resp. $i\leq n$). In case $m\leq n$ are integers, we will also denote by $\mathcal D^{[m,n]}(R)$ the full subcategory consisting of the complexes  whose nonzero homology modules are concentrated in degrees $m\leq i\leq n$.  Similarly $\D^-(R)$ (resp.  $\mathcal{D}^+(R)$) will denote the full subcategory consisting of the complexes $M$ such that $H^i(M)$, for all $i\gg0$ (resp. $i\ll0$).

\subsection{Noncommutative localization}\label{sub. localization}
Given a Grothendieck category $\mathcal{G}$ and a hereditary torsion class $\mathcal T$ in $\mathcal{G}$, there is a new Grothendieck category $\mathcal{G}/\mathcal T$, called the \emph{quotient category of $\mathcal{G}$ by $\mathcal T$},  and a \emph{quotient or localization functor} $q:\mathcal{G}\longrightarrow\mathcal{G}/\mathcal{T}$ satisfying the following two properties:

\begin{enumerate}
\item $q$ is exact and vanishes on $\mathcal T$;
\item $q$ is universal with respect to property 1. That is, if $F:\mathcal{G}\longrightarrow\mathcal{A}$ is an exact functor to any abelian category $\mathcal A$ which vanishes on $\mathcal T$, then $F$ factors through $q$ in a unique way (up to natural isomorphism).
\end{enumerate}
(see \cite{G}). Actually, by  Gabriel-Popescu's theorem (see \cite[Theorem X.4.1]{S}), all Grothendieck categories are of the form $R\Mod/\mathcal{T}$, for some ring $R$  and some hereditary torsion class $\mathcal{T}$ in $R\Mod$.  

The functor $q$ is exact and has a fully faithful right adjoint $j:\mathcal{G}/\mathcal T\longrightarrow\mathcal{G}$, called the \emph{section functor}, whose essential image is the full subcategory of $\mathcal{G}$ consisting of the objects $Y$ such that $\Hom_\mathcal{G}(T,Y)=0=\Ext_\mathcal{G}^1(T,Y)$, for all $T\in\mathcal T$. We will denote by $\mathcal G_{\mathcal{T}}$ this full subcategory and will call it the \emph{Giraud subcategory} associated to the hereditary torsion pair $(\mathcal T,\mathcal T^\perp)$.
The injective objects of $\mathcal{G}_{\mathcal T}$ are precisely the torsionfree injective objects of $\mathcal{G}$. The counit of the adjoint pair  $(q,j)$ is then an isomorphism and we will denote by $\mu :1_{\mathcal{G}}\longrightarrow j\circ q$ its unit. Note that $\Ker (\mu_M)$ and $\Coker (\mu_M)$ are in $\mathcal{T}$, for each object $M$ of $\mathcal{G}$.

When $\mathcal{G}=R\Mod$, for some  ring $R$, there is some extra information. The set of left ideals $\mathbf{a}$ of $R$ such that $R/\mathbf{a}$ is in $\mathcal{T}$ form a \emph{Gabriel topology} $\mathbf{F}=\mathbf{F}_\mathcal{T}$ (see \cite[Chapter VI]{S} for the definition), which is then a downward directed set. The $R$-module $R_\mathbf{F}:=
\varinjlim_{\mathbf{a}\in\mathbf{F}}\Hom_R(\mathbf{a},\frac{R}{t(R)})$ is then an object of $\mathcal{G}_\mathcal{T}$ and, moreover, admits a unique ring structure such that the composition of canonical maps

$$R\flecha \frac{R}{t(R)} \iso \Hom_R(R,\frac{R}{t(R)}) \flecha  R_\mathbf{F}$$
is a ring homomorphism. We  call $R_\mathbf{F}$ the \emph{ring of fractions of $R$ with respect to the hereditary torsion pair $(\mathcal{T},\mathcal{T}^\perp )$}.  Viewed as a morphism of $R$-modules this last ring  homomorphism gets identified with the unit map $\mu_R:R\longrightarrow (j\circ q)(R)$. We will simply put $\mu =\mu_R$ and it will become clear from the context when we refer to the ring homomorphism or the unit natural transformation. 
It turns out that each module in $\mathcal{G}_\mathcal{T}$ has a canonical structure of $R_\mathbf{F}$-module. Furthermore, the section functor factors in the form

$$\xymatrix{j:R\Mod/\mathcal T \ar[r]^{\hspace{0.4 cm}j'} & R_\mathbf{F}\Mod \ar[r]^{\hspace{0.15 cm}\mu_*} & R\Mod} $$
where $j'$ is fully faithful and $\mu_*$ is the restriction of scalars functor. Moreover, the functor $j'$ has an exact left adjoint $q':R_\mathbf{F}\Mod\longrightarrow R\Mod/\mathcal{T}$ whose kernel is $\mu_*^{-1}(\mathcal{T})$, i.e., the class of $R_\mathbf{F}$-modules $\tilde{T}$ such that $\mu_*(\tilde{T})\in\mathcal{T}$. This latter class is clearly a hereditary torsion class in $R_\mathbf{F}\Mod$, and then $\mathcal{G}_\mathcal{T}$ may
be also viewed as the  Giraud subcategory of $R_\mathbf{F}\Mod$ associated to the hereditary torsion pair $(\mu_*^{-1}(\mathcal{T}),\mu_*^{-1}(\mathcal{T})^\perp )$. Note that then $\mu_*$ induces an equivalence of categories $\frac{R_\mathbf{F}\Mod}{\mu_*^{-1}(\mathcal{T})} \iso \frac{R\Mod}{\mathcal{T}}$. Although the terminology is slightly changed, we refer the reader to \cite{G} and \cite[Chapters VI, IX, X]{S} for details on the concepts and their properties introduced and mentioned in this subsection. 

Note that, by the usual way of taking left and right derived functors, the above mentioned adjoint pairs induce adjoint pairs of triangulated functors $(\mathbf{L}q=q:\mathcal{D}(R)\longrightarrow\mathcal{D}(\frac{R\Mod}{\mathcal{T}}),\mathbf{R}j:\mathcal{D}(\frac{R\Mod}{\mathcal{T}})\longrightarrow\mathcal{D}(R))$ and $(\mathbf{L}q'=q':\mathcal{D}(R_\mathbf{F})\longrightarrow\mathcal{D}(\frac{R\Mod}{\mathcal{T}}),\mathbf{R}j':\mathcal{D}(\frac{R\Mod}{\mathcal{T}})\longrightarrow\mathcal{D}(R_\mathbf{F}))$. 

For simplicity,  when $R$ is Noetherian commutative and $\mathcal{T}=\mathcal{T}_Z$, where  $Z\subseteq\Spec (R)$ is a sp-subset, we will put $\mathcal{G}_Z=\mathcal{G}_{\mathcal{T}_Z}$ and $R_\mathbf{F}=R_Z$. The $R$- (resp. $R_Z$-)modules in $\mathcal{G}_Z$ will be called \emph{$Z$-closed}. It seems to be folklore that $R_Z$ is a commutative ring in this case (see \cite[Exercise IX.2]{S}). 

\section{Properties of the heart}

In the rest of the paper, unless otherwise stated, $R$ is a commutative Noetherian ring, $\phi$ will be a sp-filtration of $\Spec(R)$ and $\Ht$ will be the heart of the t-structure $(\mathcal U_\phi ,\mathcal U_\phi^\perp [1])$. We will denote by $\limite$ the direct limit in $\Ht$ and will reserve the symbol $\varinjlim$ for the direct limit in $R\Mod$. 

\subsection{Stalk complexes in the heart}

The following result is standard. We include a short proof for
completeness.

\begin{lemma} \label{lem.Ext-orthogonal}
Let $Z$ be a sp-subset of $\Spec(R)$ and let $\T_Z=\{T\in R\Mod:$
$\Supp (T)\subseteq Z\}$ be the associated hereditary torsion class
in $R\Mod$. For each integer $i\geq 0$ and each $R$-module $M$, the
following assertions are equivalent:

\begin{enumerate}
\item $\Ext_R^i(T,M)=0$, for all $T\in\T$;
\item $\Ext_R^i(R/\mathbf{p},M)=0$, for all $\mathbf{p}\in Z$.
\end{enumerate}
\end{lemma}
\begin{proof}
Put $\T:=\T_Z$ for simplicity.  
Every finitely generated module in $\T$ admits a  finite filtration
whose factors are isomorphic to modules of the form $R/\mathbf{p}$,
with $\mathbf{p}\in Z$ (see \cite[Theorem 6.4]{M}). In particular, a module $F$ is in $\T^\perp$ if, and only if, $\Hom_R(R/\mathbf{p},F)=0$, for all $\mathbf{p}\in Z$. This proves the case $i=0$. 

 For
$i>0$, we denote by $\Phi_i(1)$ and $\Phi_i(2)$ the classes of
modules $M$ which satisfy condition 1 and 2, respectively. Note that
if $i>1$ then, by dimension shifting, we have that $M\in\Phi_i(k)$
if, and only if, $\Omega^{1-i}(M)\in\Phi_1(k)$, for $k=1,2$. Hence,
the proof is  reduced to the case $i=1$. Then using the finite filtration mentioned in the previous paragraph, given $T\in\T$,
it is  easy to construct, by transfinite induction,   an ascending transfinite chain
$(T_\alpha )_{\alpha <\lambda}$ of submodules of $T$, for some ordinal $\lambda$,  such that: i)
$\bigcup_{\alpha <\lambda}T_\alpha =T$; ii) $T_\alpha
=\bigcup_{\beta <\alpha}T_\beta$, whenever $\alpha$ is a limit
ordinal $<\lambda$; iii) $T_{\alpha +1}/T_\alpha$ isomorphic to
$R/\mathbf{p}$, for some $\mathbf{p}\in Z$, whenever $\alpha
+1<\lambda$. The result is then a direct consequence of Eklof's
lemma (see \cite[3.1.2]{GT}).
\end{proof}

\begin{notation}
For each  integer $m$, we will denote by $\mathcal{TF}_{m}$ the
class of $R$-modules $Y$ such that $Y[-m]\in \Ht$.
\end{notation}

\begin{prop}\label{prop. description of stalk}
Let $m$ be an ingeter. For each $Y\in \T_{m}$, the following assertions are equivalent:
\begin{enumerate}
\item[1)] $Y\in \TF_{m};$
\item[2)] $\Ext^{k-1}_{R}(R/\mathbf{p},Y)=0$, for each integer $k\geq 1$ and for each $\mathbf{p}\in \phi(m+k);$
\item[3)]  $\Ext^{k-1}_{R}(T,Y)=0$, for each integer $k\geq 1$ and
each $T\in\T_{m+k}$.
\end{enumerate}
\end{prop}
\begin{proof}
$1)\Longleftrightarrow 2)$ Note that $Y[-m]\in\mathcal{U}_\phi$. We
then have the following chain of double implications:

\begin{center}
$Y\in\T\F_m$ $\Longleftrightarrow$ $Y[-m]\in\Ht$
$\Longleftrightarrow$ $Y[-m]\in\mathcal{U}_\phi^\perp[1]$
$\Longleftrightarrow$ $\Hom_{\D(R)}(R/\mathbf{p}[-i],Y[-m-1])=0$, for
all $i\in\mathbb{Z}$ and $\mathbf{p}\in\phi (i)$
$\Longleftrightarrow$ $\Ext_R^{i-m-1}(R/\mathbf{p},Y)=0$, for all
$i\in\mathbb{Z}$ and $\mathbf{p}\in\phi (i)$.
\end{center}
By the nonexistence of negative extensions between modules, putting
$k=i-m$ for $i>m$,  we conclude that $Y\in\T\F_m$ if, and only if,
$\Ext_R^{k-1}(R/\mathbf{p},Y)=0$, for all $k>0$ and
$\mathbf{p}\in\phi (m+k)$.

$2)\Longleftrightarrow 3)$ is a direct consequence of the previous
lemma.
\end{proof}

\begin{exems} \label{exem.TF for largest degrees}
Suppose that the filtration $\phi$ has the property that $\phi
(0)\supsetneq\phi (1)=\phi (2)=\cdots=\emptyset$. It follows from the
previous proposition that $\T\F_0=\T_0$ and
$\T\F_{-1}=\T_{-1}\cap\F_0$.
\end{exems}

Recall that, among other equivalent definitions,  an exact sequence $0\rightarrow L\stackrel{f}{\longrightarrow}M\stackrel{g}{\longrightarrow}N\rightarrow 0$ in $R\Mod$ (even if $R$ is not commutative) is called \emph{pure} when it is kept exact after applying the functor $\Hom_R(X,?):R\Mod\longrightarrow Ab$, for each finitely presented $R$-module $X$. A submodule $M'$ of $M$ is called a \emph{pure submodule} when the associated exact sequence $0\rightarrow M'\hookrightarrow M\longrightarrow M/M'\rightarrow 0$ is pure. Finally, a class $\mathcal{C}\subseteq R\Mod$ is called \emph{definable} when it is closed under taking products, direct limits and pure submodules. 

\begin{cor}\label{cor. TF closed limit}
For each integer $m$, the class $\TF_{m}$ is a definable class and it is closed under taking kernels of epimorphisms in $R\Mod$. In particular, it is also closed under taking coproducts. 
\end{cor}
\begin{proof}

Due to the fact that $R$ is a Noetherian ring,  for each
$\mathbf{p}\in \Spec(R)$, we have that all syzygies of
$R/\mathbf{p}$ are finitely presented (=finitely generated)
$R$-modules. We then  get that the functor
$\Ext^{k}_{R}(R/\mathbf{p},?)$   commutes with products and direct
limits, for each $k\in\mathbb{Z}$. 

Moreover, if $M'\subseteq M$ is a pure submodule and $p:M\twoheadrightarrow M/M'$ is the projection, then 
$\Ext_R^k(R/\mathbf{p},p):\text{Ext}_R^k(R/\mathbf{p},M)\longrightarrow \text{Ext}_R^k(R/\mathbf{p},M/M')$ is an epimorphism since so is the morphism $\Hom_R(\Omega^{k-1}(R/\mathbf{p}),p):\Hom_R(\Omega^{k-1}(R/\mathbf{p}),M)\longrightarrow\Hom_R(\Omega^{k-1}(R/\mathbf{p}),M/M')$. But, using the long exact sequence of homologies, we then get that the exact sequence

$$0\rightarrow\text{Ext}_R^k(R/\mathbf{p},M')\longrightarrow\text{Ext}_R^k(R/\mathbf{p},M)\stackrel{p_*}{\longrightarrow}\text{Ext}_R^k(R/\mathbf{p},M/M')\rightarrow 0$$
is exact, for all $\mathbf{p}\in\Spec (R)$ and all integers $k\geq 0$. 
By Proposition \ref{prop. description of stalk}, we deduce that
$\TF_m$ is a definable class in $R\Mod$. 

Let now $0 \flecha X \flecha Y \flecha Z \flecha 0$ be an exact
sequence in $R\Mod $, where $Y,Z\in\TF_m$. We have that $X\in \T_m$,
because $(\T_{m},\F_{m})$ is a hereditary torsion pair. Let us fix
$k\geq 1$ and $\mathbf{p}\in \phi(m+k)\subseteq \phi(m+k-1).$ If we
apply the long exact sequence of $\Ext(R/\mathbf{p},?)$ to the exact
sequence above, we get:

$$\xymatrix{0 \ar[r] & \Hom_{R}(R/\mathbf{p}, X) \ar[r] & \Hom_{R}(R/\mathbf{p},Y) \ar[r] & \Hom_{R}(R/\mathbf{p},Z) \ar[r] & \cdots \\ \cdots \ar[r] & \Ext^{k-2}_{R}(R/\mathbf{p},Z) \ar[r] & \Ext^{k-1}_{R}(R/\mathbf{p},X) \ar[r] & \Ext^{k-1}_{R}(R/\mathbf{p},Y) \ar[r] & \cdots}$$

From Proposition \ref{prop. description of stalk} we then get that
$\Ext^{k-2}_{R}(R/\mathbf{p},Z)=0=\Ext^{k-1}_{R}(R/\mathbf{p},Y)$,
and hence $\Ext^{k-1}_{R}(R/\mathbf{p},X)=0$. Using again Proposition \ref{prop. description of stalk}, we obtain that $X\in
\TF_{m}.$
\end{proof}

\begin{prop}\label{prop. stalk heart}
Let $m$ be any integer, let $(Y_{\lambda})$ be a direct system in
$\TF_{m}$ and $g:\coprod Y_{\lambda \mu} \flecha \coprod
Y_{\lambda}$ be the colimit-defining morphism. The following
assertions hold:
\begin{enumerate}
\item[1)] $\Ker(g)\in \TF_{m}$;
\item[2)] $\limite{(Y_{\lambda}[-m])}=(\varinjlim{Y_{\lambda}})[-m].$
\end{enumerate}
\end{prop}
\begin{proof}
Since  $\coprod Y_{\lambda \mu}, \ \coprod Y_{\lambda}$ and
$\varinjlim{Y_{\lambda}}$ are in $\TF_{m}$ it follows from Corollary
\ref{cor. TF closed limit} that $\text{Im}(g)$ and $\Ker(g)$ are
also in $\TF_m$.  We then consider the following diagram

$$\xymatrix{&&&\Ker(g)[-m+1]=H^{m-1}(\text{Cone}(g[-m]))[-m+1]\ar[d] \\ \coprod Y_{\lambda \mu}[-m] \ar[rr]^{g[-m]} & &\coprod Y_{\lambda}[-m] \ar[r] & \text{Cone}(g[-m]) \ar[d] \ar[r]^{\hspace{0.1 cm}+} & \\ &&& \varinjlim{Y_{\lambda}}[-m]=H^{m}(\text{Cone}(g[-m]))[-m] \ar[d]^{+} & \\ &&&&}$$

By the explicit construction of cokernels in $\Ht$ (see \cite{BBD}),
we deduce that $\limite{(Y_{\lambda}[-m])}=\Coker_{\Ht}(g[-m])\cong
(\varinjlim{Y_{\lambda}})[-m].$
\end{proof}

\subsection{The heart has a generator}

In this subsection we will prove that $\Ht$ always has a generator. 

\begin{notation}
For each $k\in \mathbb{Z}$, we will denoted by $\phi_{\leq k}$ the sp-filtration given by:
\begin{enumerate}
\item[1)] $\phi_{\leq k}(j):=\phi(j)$, \ \ if \ \ $j\leq k$;
\item[2)] $\phi_{\leq k}(j):=\emptyset$, \ \ \ \ \ if  \ \ $j>k.$
\end{enumerate}
\end{notation}

\begin{rem}\label{rem. truncation of phi}
For each $X\in \mathcal{H}_{\phi}$ and $k\in \mathbb{Z}$, we have that $\tau^{\leq k}(X)\in \mathcal{H}_{\phi_{\leq k+j}}$ for each $j\in \{0,1,2\}$. Indeed, it is clear that $\tau^{\leq k} (X)\in \mathcal{U}_{\phi_{\leq k+j}}$, for each $j\geq 0$. On the other hand, we consider the following triangle:
$$\xymatrix{\tau^{>k}(X)[-2] \ar[r] & \tau^{\leq k}(X)[-1] \ar[r] & X[-1] \ar[r]^{\hspace{0.3 cm}+} & }$$
Applying the cohomological functor $\Hom_{\D(R)}(R/\mathbf{p}[-m],?)$ to the previous triangle, we get that \linebreak $\Hom_{\D(R)}(R/\mathbf{p}[-m],\tau^{\leq k}(X)[-1])=0$, for each $\mathbf{p}\in \phi(m)$ and $m\leq k+2$. This shows that $\tau^{\leq k}(X)\in \mathcal{U}_{\phi_{\leq k+j}}^{\perp}[1]$ for each $j\in \{0,1,2\}$.
 \end{rem}

Let us consider now the left adjoint $L$ of the inclusion functor  $j:\Ht\monic \mathcal{U}_\phi$ (see Subsection \ref{subsection of $L$}). Due to the fully faithful condition of this latter functor, we know that the counit $L\circ j\longrightarrow 1_{\Ht}$ is a natural isomorphism.

\begin{lemma} \label{lem.from triangles to exact sequences}
Let $\xymatrix{U' \ar[r]^{f} & U\ar[r]^{g} &  Y \ar[r]^{+} & }$ be a triangle in $\D(R)$, 
where $U',U\in\mathcal{U}_\phi$ and $Y\in\Ht$. Then $\xymatrix{0 \ar[r] & L(U')\ar[r]^{L(f)} & L(U) \ar[r]^{L(g)} & Y \ar[r] & 0}$ is an exact sequence in $\Ht$. 
\end{lemma}
\begin{proof}
The octahedral axiom and the definition of $L$ give a commutative diagram in $\D(R)$

$$\xymatrix{&&&& \\ &&& M \ar@{-->}[ur]^{+} \ar[dd] & \\ & U^{'} \ar@{-->}[rru] \ar[rd] &&& \\ \tau_{\phi}^{\leq}(U[-1])[1] \ar@{-->}[ur] \ar[rr]& & U \ar[r] \ar[dr] & jL(U) \ar[r]^{+} \ar[d] &  \\ &&& j(Y) \ar[d]^{+} \ar[dr]^{+} & \\ &&&&}$$

From the right vertical triangle we deduce that $M\in\mathcal{U}_\phi^\perp[1]$. This implies that $M\cong\tau_{\phi}^{>}(U'[-1])[1]$, by using the dotted triangle. By definition of $L$, we then get that $M\in\Ht$ and $M\cong jL(U')$. Then the right vertical triangle has its three vertices in $\Ht$, which ends the proof.
\end{proof}

In the rest of the paper, we shall denote by $\mathcal C(R)$  the category of chain complexes of $R$-modules and by  $\mathcal{K}(R)$ the associated homotopy category.

\begin{prop}\label{generator of H}
Let $\mathcal{U}_{\phi}^{fg}$ be the class of all the complexes quasi-isomorphic to a complex $X$ in $\mathcal{U}_{\phi}$, such that $X^{i}$ is a finitely generated $R$-module for each $i\in \mathbb{Z}$. The isoclasses of objects $L(X)$ such that $X \in \mathcal{U}_{\phi}^{fg}$ form a set of generators of $\Ht$.
\end{prop}
\begin{proof}
It is clear that, up to isomorphism in $\Ht$, the given $L(X)$ form a set.
Let $X=(X,d)\in \mathcal{H}_{\phi}$ be any object. Our goal is to prove that it is an epimorphic image in $\Ht$ of a (set-indexed) coproduct of objects of the form $L(Y)$, with $Y\in\mathcal{U}_{\phi}^{fg}$.
In order to do that, we first consider the set $\Delta$  of all subcomplexes $Y$ of $X$ (in the abelian category
 $\C(R)$) such that $Y^k$ is finitely generated, for all $k\in\mathbb{Z}$, that  $Y\in\mathcal{U}_\phi$, when we view $Y$ as an object of $\D(R)$, and that the inclusion $\iota_Y:Y\monic X$ induces a monomorphism $H^k(Y)\monic H^k(X)$, for each $k\in\mathbb{Z}$.

We claim that $\underset{Y\in \Delta}{\cup}Y^{k}=X^{k},$ for each $k\in \mathbb{Z}$. Let us fix an integer $k$ and let $z\in X^{k}$. Put $Y^{k}:=Rz$ and $Y^{k+1}:=Rd^{k}(z)$.  We get that $Im(d^{k-1}) \cap Y^{k}$ is a finitely generated $R$-module, because $R$ is a Noetherian ring. Hence there exists a finitely generated submodule  $Y^{k-1}$ of $X^{k-1}$ such that $d^{k-1}(Y^{k-1})=Im(d^{k-1})\cap Y^{k}$. By using this argument in a recursive way, we obtain a subcomplex $Y$ of $X$ concentrated in degrees $\leq k+1$,  which has zero homology in degrees $>k$ and  the homology in degrees $\leq k$ is given by:

$$H^j(Y)=\frac{\Ker(d^{j}) \cap Y^{j}}{d^{j-1}(Y^{j-1})}=\frac{\Ker(d^{j}) \cap Y^{j}}{\Imagen(d^{j-1}) \cap Y^{j}}=\frac{\Ker(d^{j}) \cap Y^{j}}{\Imagen(d^{j-1}) \cap \Ker(d^{j}) \cap Y^{j}} \cong \frac{(\Ker(d^{j}) \cap Y^{j})+\Imagen(d^{j-1})}{\Im(d^{j-1})} \subseteq \frac{\Ker(d^{j})}{\Imagen(d^{j-1})}=H^j(X).$$
It follows that $Y \in \Delta$ since  $(\mathcal{T}_{j},\mathcal{F}_{j})$ is a hereditary torsion pair, for each integer $j$. Our claim follows due to the fact that $z\in Y^{k}.$ \\

Consider now the epimorphism $p:\underset{Y\in\Delta}{\coprod} Y\longrightarrow X$ in the abelian category $\C(R)$ whose components are the inclusions $\iota_Y:Y\monic X$ and the corresponging exact sequence $0\rightarrow N\longrightarrow\underset{Y\in\Delta}{\coprod} Y\stackrel{p}{\longrightarrow} X\rightarrow 0$ in this category. From the construction given in the previous paragraph we easily deduce that the induced map $H^k(p):H^k( \underset{Y\in\Delta}{\coprod}Y)\cong \underset{Y\in\Delta}{\coprod}H^k(Y)\longrightarrow H^k(Z)$ is surjective, for all $k\in\mathbb{Z}$. When we consider the associated triangle $\xymatrix{N \ar[r] &  \underset{Y\in\Delta}{\coprod}Y \ar[r]^{\hspace{0.3 cm}p} & X \ar[r]^{+} & }$ in $\D(R)$ and take the long exact sequence of homologies, we get that the sequence
$$\xymatrix{0\ar[r] & H^k(N) \ar[r] & \underset{Y\in\Delta}{\coprod}H^k(Y)\ar[r]^{\hspace{0.4 cm}H^k(p)}& H^k(X) \ar[r] & 0 }$$ is exact, for each $k\in\mathbb{Z}$. It follows that $H^k(N)\in\T_k$, for each $k\in\mathbb{Z}$, so that $N$ is in $\mathcal{U}_\phi$. By applying Lemma \ref{lem.from triangles to exact sequences} and using the fact that a left adjoint functor preserves coproducts, we get a short exact sequence in $\Ht$  
$$\xymatrix{0 \ar[r] & L(N) \ar[r] &  \underset{Y\in\Delta}{\coprod}L(Y) \ar[r] & X\ar[r] & 0.}$$ This ends the proof since $\Delta\subset\mathcal{U}_{\phi}^{fg}$.
\end{proof}

\subsection{The AB5 and modular conditions via localization at prime ideals}

Let $\mathbf{p}$ be a prime ideal of $R$ and let $f:R \flecha R_{\mathbf{p}}$ the canonical ring homomorphism. The exactness of both the extension and restriction of scalars functors with respect to this morphism, gives   triangulated functors $f_*=\mathbf{R}f_*:\D(R_\mathbf{p})\flecha \D(R)$ and  $f^*=\text{L}f^{*}:D(R) \flecha D(R_{\mathbf{p}})$ such that $(f^*,f_*)$ is an adjoint pair. Moreover $f_*$  is fully faithful since  $R_{\mathbf{p}}$ is a flat module, and hence $R\flecha R_\mathbf{p}$ is a homological epimorphism (see \cite[Section 4]{GL} and  \cite[Section 4]{NS}). In particular, the counit $f^*\circ f_*\flecha 1_{\D(R_\mathbf{p})}$ is a natural isomorphism.\\

All throughout the paper, whenever necessary, we view $\Spec(R_{\mathbf{p}})$ as the subset of $\Spec(R)$ consisting of the $\mathbf{q}\in \Spec(R)$ such that $\mathbf{q}\subseteq \mathbf{p}$. The following result shows that hearts of t-structures behave rather well with respect to localization at primes.

\begin{prop} \label{prop.localization of hearts}
Let $\phi$ be an sp-filtration of $\Spec(R)$, let $\mathbf{p}$ be a prime ideal of $R$ and let  $\phi_{\mathbf{p}}$ denote the sp-filtration of $\Spec(R_{\mathbf{p}})$, defined by $\phi_{\mathbf{p}}(i):=\phi(i)\cap \Spec(R_{\mathbf{p}})$, for each $i\in \mathbb{Z}$. Let $\Ht$ and $\mathcal{H}_{\phi_\mathbf{p}}$ denote the hearts of the associated t-structures in $\D(R)$ and $\D(R_\mathbf{p})$, respectively. The functors $f^*$ and $f_*$ induce by restriction an adjoint pair of exact functors $(f^*:\Ht \flecha \mathcal{H}_{\phi_\mathbf{p}},f_*:\mathcal{H}_{\phi_\mathbf{p}}\flecha\Ht )$ whose counit is a natural isomorphism (equivalently, $f_*$ is fully faithful).
\end{prop}
\begin{proof}
By  \cite[Proposition 2.9]{AJS}, we have that $f^*(\Ht )\subseteq\mathcal{H}_{\phi_\mathbf{p}}$ and $f_*(\mathcal{H}_{\phi_\mathbf{p}})\subseteq\Ht$. The proposition follows immediately from this since a short exact sequence in the heart of a t-structure is the same as a triangle in the ambient triangulated category which has its three vertices in the heart. 
\end{proof}

\begin{cor}\label{cor. Hp is AB5}
The following assertions are equivalent:
\begin{enumerate}
\item $\Ht$ is an AB5 abelian category;
\item $\mathcal{H}_{\phi_{\mathbf{p}}}$ is an AB5 abelian category, for all $\mathbf{p}\in \Spec(R)$.
\end{enumerate}
\end{cor}
\begin{proof}

1) $\Longrightarrow$ 2) Let us fix a $\mathbf{p}\in \Spec(R)$ and let us consider a direct system
$$\xymatrix{0 \ar[r] & X_{\lambda} \ar[r] & Y_{\lambda} \ar[r] & Z_{\lambda} \ar[r] & 0}$$
of short exact sequences in $\mathcal{H}_{\phi_{\mathbf{p}}}$. By applying $f_*$ and taking direct 
limits in $\Ht$, due to the AB5 condition of this latter category,  we get an exact sequence in $\Ht$
\begin{center}
$\xymatrix{0\ar[r] & \limite f_*(X_\lambda ) \ar[r]&  \limite f_*(Y_\lambda) \ar[r] & \limite f_*(Z_\lambda ) \ar[r] & 0}. $
\end{center}
But the functor $f^*:\Ht \flecha \mathcal{H}_{\phi_\mathbf{p}}$ preserves direct limits since it is a left adjoint. This, together with the fact that the 
 counit $f^*\circ f_*\flecha 1_{ \mathcal{H}_{\phi_\mathbf{p}}}$ is an isomorphism, imply that the sequence
\begin{center}
$\xymatrix{0 \ar[r] & \varinjlim_{\mathcal{H}_{\phi_\mathbf{p}}}X_\lambda \ar[r] & \varinjlim_{\mathcal{H}_{\phi_\mathbf{p}}}Y_\lambda \ar[r] & \varinjlim_{\mathcal{H}_{\phi_\mathbf{p}}}Z_\lambda \ar[r] & 0 }$
\end{center}
is exact in $\mathcal{H}_{\phi_\mathbf{p}}$. \\

2) $\Longrightarrow$ 1) Let us consider a direct system of short exact sequences in $\Ht$
$$\xymatrix{0 \ar[r] & X_{\lambda} \ar[r] & Y_{\lambda} \ar[r] & Z_{\lambda} \ar[r] &0}.$$
By right exactness of colimits, we  get an exact sequence in $\Ht$:
$$\xymatrix{\limite{X_{\lambda}} \ar[r]^{h} & \limite{Y_{\lambda}} \ar[r]^{g} & \limite{Z_{\lambda}} \ar[r] &0}.$$
Consider now any prime ideal $\mathbf{p}$ and the associated ring homomorphism $f:R\flecha R_\mathbf{p}$. By the exactness of $f^*$ and the fact that this functor preserves direct limits, we get a commutative diagram with exact rows  

$$\xymatrix{& f^{*}(\limite X_{\lambda}) \ar[r] \ar[d]^{\wr} & f^{*}(\limite Y_{\lambda}) \ar[d]^{\wr} \ar[r] & f^{*}(\limite Z_{\lambda}) \ar[r] \ar[d]^{\wr} & 0 \\ 0 \ar[r] & \varinjlim_{\mathcal{H}_{\phi_{\mathbf{p}}}} f^{*}(X_{\lambda}) \ar[r] & \varinjlim_{\mathcal{H}_{\phi_{\mathbf{p}}}} f^{*}(Y_{\lambda}) \ar[r] & \varinjlim_{\mathcal{H}_{\phi_{\mathbf{p}}}} f^{*}(Z_{\lambda}) \ar[r] & 0}$$

where the vertical arrows are isomorphisms. It follows $ R_\mathbf{p} \otimes \Ker_{\Ht}(h)
=f^*(\Ker_{\Ht}(h))=0$ in $\D(R_\mathbf{p})$, for all $\mathbf{p}\in\Spec(R)$. This implies that the complex $\Ker_{\Ht}(h)$ is acyclic, and hence it is  zero in $\Ht$. 
\end{proof}

\begin{cor} \label{cor.modular heart via localization}
If $\Ht$ is a module category, then $\mathcal{H}_{\phi_\mathbf{p}}$ is a module category, for each $\mathbf{p}\in\Spec (R)$. 
\end{cor}
\begin{proof}
Note that both hearts are AB3 (=cocomplete) abelian categories (see \cite[Proposition 3.2]{PS}). 

Let $P$ be a progenerator of $\Ht$. Fix a prime ideal $\mathbf{p}$ and consider the canonical ring homomorphism $f:R\flecha R_\mathbf{p}$. We will prove that $f^*(P)$ is a progenerator of $\mathcal{H}_{\phi_\mathbf{p}}$. Indeed $f^*(P)$ is a projective object of this category since $f^*$ is left adjoint of an exact functor, which implies that it preserves projective objects. On the other hand, if $Y$ is any object of $\mathcal{H}_{\phi_\mathbf{p}}$, then we have an epimorphism $p:P^{(\Lambda )}\epic f_*(Y)$ in $\Ht$, for some set $\Lambda$. Exactness of $f^*$ and the fact that the counit of the adjunction $(f^*,f_*)$ is an isomorphism imply that $f^*(p):f^*(P)^{(\Lambda )}\cong f^*(P^{(\Lambda )})\flecha f^*f_*(Y)\cong Y$ is an epimorphism in $\mathcal{H}_{\phi_\mathbf{p}}$. Then $f^*(P)$ is a generator of $\mathcal{H}_{\phi_\mathbf{p}}$.\\

We finally prove that $f^*(P)$ is compact. Let $(Y_\lambda )_{\lambda\in\Lambda}$ be a family of objects of $\mathcal{H}_{\phi_\mathbf{p}}$. Using the adjunction $(f^*,f_*)$ and the fact that $f_*:\mathcal{H}_{\phi_\mathbf{p}}\flecha \Ht$ preserves coproducts, we immediately get that the canonical morphism
$$\coprod_{\lambda\in\Lambda}\Hom_{\mathcal{H}_{\phi_\mathbf{p}}}(f^*(P),Y_\lambda )\flecha  \Hom_{\mathcal{H}_{\phi_\mathbf{p}}}(f^*(P),\coprod_{\lambda\in\Lambda}Y_\lambda )$$ is an isomorphism. 
\end{proof}

\section{The Grothendieck condition for left bounded filtrations}
\subsection{Some technical auxiliary results}
All throughout this section, $\phi$ is an sp-filtration which is \emph{left bounded}, i.e., there is a $n$ integer such that $\phi(n)=\phi(n-k)$, for all integers $k\geq0$. We fix such an $n$, so that $\phi$ will be of the form 
$$\cdots =\phi(n-1)=\phi(n)\supseteq \phi(n+1) \supseteq \phi(n+2) \supseteq \cdots$$

\begin{lemma}\label{lem. bounded heart}
For each $X\in \Ht$, the following assertions holds.
\begin{enumerate}
\item[1)] $H^{j}(X)=0$, for each integer $j<n.$
\item[2)] If $\phi(m)=\emptyset$, for some integer $m$, then $H^{j}(X)=0$, for each $j\geq m$.
\end{enumerate}
\end{lemma}
\begin{proof}
Assertion 2 is clear, so that we just prove 1. 
By the left bounded condition of $\phi$, we know that $H^{k}(X)\in \T_{n}$, for all integers $k$.  It follows from \cite[Theorem 1.8]{AJS} that $\mathbf{R}\Gamma_{\phi(n)}X\cong X$. On the other hand, $X[-1]\in \U^{\perp}$, and hence $\mathbf{R}\Gamma_{\phi(k)}X\in \D^{>k-1}(R)$ for all integers $k$. Putting $k=n$ we get that $\mathbf{R}\Gamma_{\phi(n)}X=X\in D^{>n-1}(R)$.
\end{proof}

\begin{lemma}\label{lem. last homology}
For each $X\in \Ht$, the following assertions holds.
\begin{enumerate}
\item[1)] $\Hom_{\D(R)}(Y[-k],\tau^{\leq k}(X))=0$, for all $k\in \mathbb{Z}$ and for each $Y\in \T_{k+1}$.
\item[2)] $\Hom_{\D(R)}(Y[-k],\tau^{\leq k-2}(X)[-1])=0$, for all $k\in \mathbb{Z}$ and for each $Y\in \T_{k}$.
\item[3)] If $m$ is the smallest integer such that $H^{m}(X)\neq 0$, then
$$H^{m}(X)\in \F_{m+1} \bigcap \Ker \Ext^{1}_{R}(\T_{m+2},?).$$
\end{enumerate}
\end{lemma}
\begin{proof}
1) and 2) follow from the fact that $\tau^{\leq k}(X)\in \mathcal{H}_{\phi_{\leq k+1}}$ and $\tau^{\leq k-2}(X)\in \mathcal{H}_{\phi_{\leq k}}$ for all integer $k$ (see Remark \ref{rem. truncation of phi}.) \\

3) Let $m$ the smallest integer, such that $H^{m}(X)\neq 0$. By 1), we have:
$$0=\Hom_{D(R)}(Y[-m],\tau^{\leq m}(X))=\Hom_{\D(R)}(Y[-m],H^{m}(X)[-m])=\Hom_{R}(Y,H^{m}(X))$$
for each $Y\in \T_{m+1}$, and hence $H^{m}(X)\in \F_{m+1}.$ Similarly, by 1) and 2),  we have
\begin{small}
$$0=\Hom_{\D(R)}(Y[-m-2],\tau^{\leq m+2-2}(X)[-1])=\Hom_{\D(R)}(Y[-m-2],H^{m}(X)[-m-1])=\Ext^{1}_{R}(Y,H^{m}(X))$$
for each $Y\in \T_{m+2}$.
\end{small}
\end{proof}

In the following proposition, we gather a few properties that will  be 
needed later on.

\begin{prop} \label{prop.gathering properties for bounded}
Let $\phi :\mathbb{Z}\longrightarrow\mathcal{P}(\Spec(R))$ be a left
bounded sp-filtration and let us put $m=\text{min}\{i\in\mathbb{Z}:$
$\phi (i)\supsetneq\phi (i+1)\}$. Consider a direct system
$(X_\lambda )_{\lambda\in\Lambda}$ in $\Ht$ and let
$\varphi_j:\varinjlim H^j(X_\lambda )\flecha H^j(\limite
X_\lambda )$ be the canonical map, for each $j\in\mathbb{Z}$. The
following assertions hold:

\begin{enumerate}
\item If $\phi$ is also eventually trivial and $t=\text{max}\{i\in\mathbb{Z}:$ $\phi
(i)\neq\emptyset\}$, then $\varphi_t$ is an isomorphism and
$\varphi_{t-1}$ is an epimorphism;
\item $\Supp(\Ker (\varphi_j))\subseteq\phi (j+1)$ and $\Supp(\Coker(\varphi_j))\subseteq\phi
(j+2)$, for all $j\in\mathbb{Z}$;
\item $\varphi_m$ is an isomorphism.
\end{enumerate}
\end{prop}
\begin{proof}
1) By taking the exact sequence of homologies associated to any
exact sequence in $\Ht$, one readily sees that the functor
$H^t:\Ht\flecha R\Mod$ is right exact. Since it preserves
coproducts we conclude that it preserve colimits and, in particular,
direct limits. On the other hand, let us consider the canonical
exact sequence

\begin{center}
$0 \flecha K \xymatrix{\ar[r]^{\iota}&} \coprod{X_{\lambda}}
\xymatrix{\ar[r]^{p} &} \limite{X_{\lambda}} \flecha 0$.
\end{center}
By \cite[Lemma 3.5]{CGM}, we get that $H^{t}(\iota)$ is a
monomorphism. But, applying the exact sequence of homologies to the
exact sequence above, we obtain the following exact sequence in
$R\Mod$
\begin{center}
$\xymatrix{\cdots \ar[r] & H^{t-1}(\coprod {X_{\lambda}})
\ar[r]^{H^{t-1}(p)\hspace{0.2cm}} & H^{t-1}(\limite{X_{\lambda}})
\ar[r] & H^{t}(K) \ar@{^(->}[r]^{H^{t}(\iota)\hspace{0.2cm}} &
H^{t}(\coprod{X_{\lambda}}) \ar[r] & \cdots}$
\end{center}
Thus $H^{t-1}(p)$ is an epimorphism and it is the
composition
\begin{center}
$\coprod H^{t-1}(X_{\lambda}) \epic \varinjlim{H^{t-1}(X_{\lambda})}
\xymatrix{\ar[r]^{can} &} H^{t-1}(\limite{X_\lambda})$
\end{center}
since  $H^{t-1}$ preserves coproducts. Therefore the second
arrow is also an epimorphism. \\

2) Let $\mathbf{p}\in\Spec (R)\setminus\phi (j+2)$. Then the
sp-flitration $\phi_\mathbf{p}$ of $\Spec(R_\mathbf{p})$ satisfies
that $\phi_\mathbf{p}(j+2)=\emptyset$. Applying assertion 1 to
$\phi_\mathbf{p}$, we get that
$R_\mathbf{p}\otimes_R\Coker(\varphi_j)=0$. Similarly, if
$\mathbf{p}\in\Spec (R)\setminus\phi (j+1)$ then assertion 1 for
$\phi_\mathbf{p}$ gives that $R_\mathbf{p}\otimes\Ker
(\varphi_j)=0$. \\

3) By Lemmas \ref{lem. last homology} and \ref{lem.Ext-orthogonal}
and the fact that $\Ext_R^i(R/\mathbf{p},?)$ preserves direct
limits, for each $\mathbf{p}\in\Spec(R)$, we know that
$\varinjlim{H^{m}(X_{\lambda})}$ is in $\F_{m+1}\cap\Ker
\Ext^{1}_{R}(\T_{m+2},?)$. This in turn implies that
$\Ker(\varphi_m)\in\F_{m+1}$.  But assertion 2 tells us that $\Ker
(\varphi_{m})\in\T_{m+1}$, so that $\varphi_m$ is a monomorphism.
Assertion 2 also says that $\Coker(\varphi_m)\in\T_{m+2}$, which
implies then that  the sequence

\begin{center}
$0\flecha
\varinjlim{H^{m}(X_{\lambda})}\xymatrix{\ar[r]^{\varphi_m} & } H^m(\limite
X_\lambda )\flecha \Coker(\varphi_m) \flecha 0$
\end{center}
splits. By Lemma \ref{lem. last homology}, we know that $H^m(\limite
X_\lambda )\in\F_{m+1}$, which implies then  that
$\Coker(\varphi_m)\in\F_{m+1}$. But we also have that
$\Coker(\varphi_m)\in\T_{m+2}\subseteq\T_{m+1}$, so that $\varphi_m$
is an epimorphism.
\end{proof}

\subsection{The eventually trivial  case}

All throughout this subsection, $\phi$ is a left bounded and eventually trivial sp-filtration. $L(\phi)=\text{min}\{i \in \mathbb{Z}| \phi(i)=\emptyset \}-\text{min}\{i \in \mathbb{Z} | \phi(i)\supsetneq \phi(i+1)\}$ is called the \emph{length of the sp-filtration}. Without loss of generality, we assume that  $\phi$ is of the form

$$\cdots =\phi(-n-1)=\phi(-n)\supsetneq \phi(-n+1) \supseteq \cdots \supseteq \phi(-1) \supseteq \phi(0) \supsetneq \phi(1)=\emptyset = \cdots,$$
for some $n\in\mathbb{N}$.
 In such case, we have that $\Ht \subset \D^{[-n,0]}(R)$ (see Lemma \ref{lem. bounded
 heart}) and $L(\phi )=n+1$.

We start by giving a series of auxiliary lemmas which prepare the ground for the induction argument that will prove that $\Ht$ is a Grothendieck category in this case. 

\begin{lemma} \label{lem. T[0] closed for subobjects}
$\T_0[0]$ is a full subcategory of $\Ht$ closed under taking
subobjects.
\end{lemma}
\begin{proof}
By Example \ref{exem.TF for largest degrees}, we know that $\T_0[0]$
is a full subcategory of $\Ht$. If now $u:X\monic T[0]$ is
a monomorphism in $\Ht$, where $T\in\T_0$, then truncation with
respect to the canonical t-structure \linebreak $(\D^{\leq -1}(R),\D^{\geq
-1}(R))$ gives an isomorphism
$$\text{Hom}_{\D(R)}(\tau^{\geq
0}X,T[0])\cong\text{Hom}_{\D(R)}(H^0(X)[0],T[0]) \iso \text{Hom}_{\Ht}(X,T[0])\cong\text{Hom}_{\D(R)}(X,T[0]) .$$ Bearing in mind
that $\tau^{\geq 0}X\cong H^0(X)[0]$ is in $\Ht$, we have a
factorization
$\xymatrix{u:X \ar[r]^{v \hspace{0.5 cm}} & H^0(X)[0] \ar[r] & T[0]}$ in
$\Ht$, where $v$ gets identified with the canonical morphism
$X\flecha \tau^{\geq 0}X\cong\tau^{>-1}X$. We then get that
$v$ is a monomorphism in $\Ht$. But Remark \ref{rem. truncation of
phi} says that $\tau^{\leq -1}X\in\Ht$, which implies that
$\Ker_{\Ht}(v)=\tau^{\leq -1}X$. This last complex is then zero in
$\Ht$, which implies that $X\cong H^0(X)[0]\in\T_0[0]$.
\end{proof}

\begin{lemma}\label{lem. comparate sp-filtration}
If $(X_{\lambda})$ is a direct system in $\Ht \cap \D^{\leq -1}(R)$, then $(X_{\lambda})$ is a direct system of $\mathcal{H}_{\phi_{ \leq -1}}$. Moreover, there exists a $T\in\T_{0}$ and a triangle in $\D(R)$ of the form:
$$\xymatrix{T[1] \ar[r] & \varinjlim_{\mathcal{H}_{\phi_{ \leq -1}}}{X_{\lambda}} \ar[r] & \limite{X_{\lambda}} \ar[r]^{\hspace{0.6cm}+} &}$$
\end{lemma}
\begin{proof}
From Remark \ref{rem. truncation of phi}, we get that $(\tau^{\leq -1}(X_{\lambda}))=(X_{\lambda})$ is a direct system of $\mathcal{H}_{\phi_{\leq -1}}$. We consider the associated triangle given by the colimit-defining morphism:

$$\xymatrix{\coprod{X_{\lambda \mu}} \ar[r]^{g} & \coprod {X_{\lambda}} \ar[r] &Z \ar[r]^{+} &}$$
and let $\mathcal{U}$ denote either $\mathcal{U}_\phi$ or $\mathcal{U}_{\phi_{\leq -1}}$ in this paragraph and,  correspondingly, let $\mathcal{H}$ denote either $\mathcal{H}_\phi$ or $\mathcal{H}_{\phi_{\leq -1}}$. 
Since $\mathcal{H} =\mathcal{U}\cap\mathcal{U}^\perp [1]$ is closed under coproducts and we have inclusions $\mathcal{U} [1]\subseteq\mathcal{U}$ and $\mathcal{U}^\perp [1]\subseteq\mathcal{U}^\perp [2]$, it immediately follows that $Z\in\mathcal{U}\cap\mathcal{U}^\perp [2]$. Consider next the following triangle, given by truncation with respect to the t-structure $(\mathcal{U},\mathcal{U}^\perp [1])$:
$$\tau_\mathcal{U}(Z[-1])[1]\longrightarrow Z\longrightarrow\tau^{\mathcal{U}^\perp}(Z[-1])[1]\stackrel{+}{\longrightarrow} $$
We get that $\tau_\mathcal{U}(Z[-1])[1]\in\mathcal{U}[1]\cap\mathcal{U}^\perp [2]=\mathcal{H}[1]$ and $\tau^{\mathcal{U}^\perp}(Z[-1])[1]\in\mathcal{U}\cap\mathcal{U}^\perp [1]=\mathcal{H}$. By \cite{BBD}, we conclude that $\tau^{\mathcal{U}^\perp}(Z[-1])[1]=\text{Coker}_\mathcal{H}(g)=\varinjlim_\mathcal{H}X_\lambda$.

In particular, we get that $\limite{X_{\lambda}}=\tau_{\phi}^{>}(Z[-1])[1]$ and $\varinjlim_{\mathcal{H}_{\phi_{\leq -1}}}{X_{\lambda}}=\tau^{>}_{\phi_{\leq -1}}(Z[-1])[1]$. On the other hand, by \cite[Corollary 5.6]{AJS}, we have the following triangle:
$$\xymatrix{H^{0}(\tau_{\phi}^{\leq}(Z[-1]))[0] \ar[r] & \tau_{\phi_{\leq -1}}^{>}(Z[-1]) \ar[r] & \tau_{\phi}^{>}(Z[-1]) \ar[r]^{\hspace{0.7 cm}+} & }$$
The result follows by putting $T=H^{0}(\tau_{\phi}^{\leq}(Z[-1]))$ since  $\tau_{\phi}^{\leq}(Z[-1])\in \U$.
\end{proof}
\begin{lemma} \label{lem.domino-effect2}
Let $f=(f_\lambda :X_\lambda \flecha Y_\lambda
)_{\lambda\in\Lambda}$ be a morphism between the direct systems
$(X_\lambda )_{\lambda\in\Lambda}$ and $(Y_\lambda
)_{\lambda\in\Lambda}$ of $\D(R)$, and suppose that there are
integers $m\leq r$ such that $X_\lambda ,Y_\lambda\in \D^{[m,r]}(R)$,
for all $\lambda\in\Lambda$. If the induced map $\varinjlim
H^j(X_\lambda )\flecha \varinjlim H^j(Y_\lambda )$ is an
isomorphism, for all $m\leq j\leq r$, then the induced map
$\varinjlim \Hom_{\D(R)}(C[s],X_\lambda )\flecha \varinjlim
\Hom_{\D(R)}(C[s],Y_\lambda )$ is an isomorphism, for each finitely generated $R$-module 
$C$ and each integer $s$.
\end{lemma}
\begin{proof}
 Let $C$ be any finitely generated $R$-module. Due to the hypothesis
and the fact that the  functor $\Ext_R^i(C,?):R\Mod \flecha 
Ab$ preserves direct limits, for  each
$i\geq 0$, we have isomorphisms

\begin{center}
$\varinjlim\Hom_{\D(R)}(C[s],H^j(X_\lambda )
[-j])\cong\Hom_{\D(R)}(C[s],(\varinjlim H^j(X_\lambda))
[-j])\stackrel{\cong}{\longrightarrow}\Hom_{\D(R)}(C[s],(\varinjlim
H^j(Y_\lambda)) [-j])\cong \varinjlim\Hom_{\D(R)}(C[s],H^j(Y_\lambda
) [-j])$,
\end{center}
for all  $s,j\in\mathbb{Z}$.\\

Each $X_\lambda$ is an iterated finite extension of $H^m(X_\lambda
)[-m]$, $H^{m+1}(X_\lambda )[-m-1],\dots,H^r(X_\lambda )[-r]$ and
similarly for $Y_\lambda$. Using this and the previous paragraph,
one easily proves by induction on $k\geq 0$, that the induced map

\begin{center}
$\varinjlim\Hom_{\D(R)}(C[s],\tau^{\leq m+k}(X_\lambda)
)\flecha \varinjlim\Hom_{\D(R)}(C[s],\tau^{\leq m+k}(Y_\lambda)
)$
\end{center}
is an isomorphism, for all  $s\in\mathbb{Z}$.
The case $k=r-m$ gives the desired result.
\end{proof}

\begin{lemma} \label{lem.Yoneda lemma}
Let $(X_\lambda )_{\lambda}$ be a direct system in $\Ht$, let $L'$
be any object of $\D(R)$, which is viewed as a constant
$\Lambda$-direct system, and let $(f_\lambda
:X_\lambda\flecha  L')_{\lambda\in\Lambda}$ be a morphism of
direct systems in $\D(R)$. If the induced morphism $\varinjlim
H^j(f_\lambda ):\varinjlim H^j(X_\lambda )\flecha H^j(L')$
is an isomorphism, for all $j\in\mathbb{Z}$, then $L'$ is in $\Ht$.
\end{lemma}
\begin{proof}
For each $j\in\mathbb{Z}$, we  have that $H^j(L')\in\T_j$ since
$\T_j$ is closed under direct limits in $R\Mod $. It follows that
$L'\in\mathcal{U}_\phi$. On the other hand, by the previous lemma
and the fact that $X_\lambda\in \D^{[-n,0]}(R)$ for each
$\lambda\in\Lambda$, we know that the canonical map

\begin{center}
$\varinjlim \Hom_{\D(R)}(C[s],X_\lambda )\flecha 
\Hom_{\D(R)}(C[s],L')$
\end{center}
is an isomorphism, for each finitely generated module $C$ and each
$s\in\mathbb{Z}$. When taking $s=-j+1$ and $C=R/\mathbf{p}$, with
$\mathbf{p}\in\phi (j)$, we get that
$\Hom_{\D(R)}(R/\mathbf{p}[-j],L'[-1])=0$, for all $j\in\mathbb{Z}$
and all $\mathbf{p}\in\phi (j)$. Therefore we get that
$L'[-1]\in\mathcal{U}_\phi^\perp$, so that
$L'\in\mathcal{U}_\phi\cap\mathcal{U}_\phi^\perp [1]=\Ht$.
\end{proof}

The crucial step for our desired induction argument is the following. 

\begin{lemma}\label{lem. domino effect}
If $(X_{\lambda})$ is a direct system in $\Ht \cap \D^{\leq k+1}(R)$
such that the canonical morphism
\begin{center}
$\varphi_{j}:\varinjlim{H^{j}(X_{\lambda}) \flecha H^{j}(\limite{X_{\lambda}})}$
\end{center}
is an isomorphism, for all $j\leq k$, then $\varphi_{j}$ is an isomorphism, for all $j\in \mathbb{Z}$.
\end{lemma}
\begin{proof}
Let $g:\coprod X_{\lambda \mu} \flecha \coprod X_{\lambda}$ be the
colimit-defining morphism and put $L:=\limite{X_{\lambda}}$. By
\cite{BBD}, there exist a $M\in \Ht$ (in fact $M=\Ker_{\Ht}(g)$) and
a diagram of the following form, where the vertical and horizontal sequences
are both triangles:

$$\xymatrix{&& M[1]\ar[d] \\ \coprod{X_{\lambda \mu}} \ar[r]^{g} & \coprod{X_{\lambda}} \ar[r] & Z \ar[r]^{+} \ar[d] & \\ && L \ar[d]^{+} \\ &&& }.$$
By using the octahedral axiom, we then get another commutative
diagram, where rows and columns are triangles:

$$\xymatrix{\tau^{\leq k+1}(M)[1] \ar[r] \ar@{=}[d] & M[1] \ar[r] \ar[d] & \tau^{>k+1}M[1] \ar[r]^{\hspace{0.7 cm}+} \ar[d] & \\ \tau^{\leq k+1}(M)[1] \ar[r] \ar[d] & Z \ar[r] \ar[d] & L^{'} \ar[r]^{+} \ar[d] & \\ 0 \ar[d]^{+} \ar[r] & L \ar@{=}[r] \ar[d]^{+} & L\ar[r]^{+} \ar[d]^{+} & \\&&&&}$$

We now use the exact sequences of homologies associated to the
different triangles appearing in these diagrams. When applied to the
horizontal triangle in the first diagram, we get an isomorphism
$\varinjlim H^{k+1}(X_\lambda )\cong\Coker
(H^{k+1}(g))\iso H^{k+1}(Z)$ and get
that $Z\in \D^{\leq k+1}(R)$. When applied to the central horizontal
triangle in the second diagram, we get another isomorphism
$H^{k+1}(Z)\iso H^{k+1}(L')$ and we also
get that $L'\in \D^{[-n-1,k+1]}(R)$. But, when applied to the right
vertical triangle of the second diagram, we get that
$H^{-n-1}(L')=0$ since $L\in\Ht\subseteq \D^{[-n,0]}(R)$ (see Lemma
\ref{lem. bounded heart}). It follows that  $L'\in \D^{[-n,k+1]}(R)$.\\

Note that, when viewing $Z$ and $L'$ as constant $\Lambda$-direct
systems in $\D(R)$, we get morphisms of direct systems $(X_\lambda
\longrightarrow Z\longrightarrow L')$, which yield a morphism
$\psi_j:\varinjlim H^j(X_\lambda)\longrightarrow H^j(L')$ in $\Mod
R$, for each $j\in\mathbb{Z}$. We claim that $\psi_j$ is an
isomorphism, for all $j\in\mathbb{Z}$. The previous paragraph proves
the claim for $j=k+1$ and the fact that $X_\lambda ,L'\in
\D^{[-n,k+1]}(R)$ reduces the proof to the case when
$j\in\{-n,-n+1,\dots,k\}$. But, when we apply the exact sequence of
homologies to the right vertical triangle of the second diagram
above, we get an isomorphism $H^j(L')\cong H^j(L)$, for $-n\leq
j<k$, and a monomorphism $H^k(L')\monic H^k(L)$. The
hypothesis of the present lemma then gives that $\psi_j$ is an
isomorphism, for all $j<k$, and a composition of morphisms
$\varinjlim H^k(X_\lambda)\flecha  H^k(L')\flecha 
H^k(L)$, which is an isomorphism and whose second arrow is a
monomorphism. It then follows that both arrows in this composition
are isomorphisms, thus settling our claim.\\

Once we know that $\psi_j$ is an isomorphism, for all
$j\in\mathbb{Z}$, Lemma \ref{lem.Yoneda lemma} says that $L'\in\Ht$.
Then we have that $L'[-1]\in\mathcal{U}_\phi^\perp$ and, using the
central horizontal triangle of the second diagram above and the fact
that $Z\in\mathcal{U}_\phi^\perp [2]$, we deduce that $\tau^{\leq
k+1}M\in\mathcal{U}_\phi^\perp [1]$. It follows that $\tau^{\leq
k+1}M\in\Ht$ since, due to the fact that $M\in\Ht$, we have that
$\Supp (H^j(\tau^{\leq k+1}M))\subset\phi (j)$, for all
$j\in\mathbf{Z}$. Using now the explicit calculation of
$\Coker_{\Ht} (g)$ (see \cite{BBD}), we then get that
$L'\cong\Coker_{\Ht}(g)=L$ in $\Ht$.
\end{proof}

\begin{teor}\label{teo. AB5 finite}
Let  $\phi$ be a left bounded eventually trivial sp-filtration of $\Spec (R)$. The classical cohomological
functor $H^{k}:\Ht \flecha R\Mod$ preserves direct limits, for each
integer $k$. In particular, $\Ht$ is an AB5 abelian category.
\end{teor}
\begin{proof}
The final part of the statement follows from \cite[Proposition
3.4]{PS}. Without loss of generality, if $L(\phi )=n+1$ we will
assume that $\phi$ is the filtration $\cdots =\phi (-n-k)=\cdots=\phi
(-n)\supsetneq\phi (-n+1)\supseteq\phi
(-n+2) \cdots \supseteq\phi(0)\supsetneq\phi (-1)= \cdots=\emptyset$. The
proof  will be done   by induction on the length $L(\phi )=n+1$ of
the filtration.  The cases when $L(\phi )=0,1$ are covered by
Proposition \ref{prop.gathering properties for bounded}.  We then
assume in the sequel that $n\geq 1$, that $L(\phi )=n+1$ and that
the result is true for left bounded eventually trivial sp-filtrations of length $\leq n$.

It is important for the reader to note that if $X\in\mathcal{H}_\phi$ then, by Remark \ref{rem. truncation of phi}, we have that $\tau^{\leq 2}X$ and $\tau^{\leq -1}X$ are in $\mathcal{H}_\phi\cap\mathcal{H}_{\phi_{\leq -1}}$. Moreover, by Example \ref{exem.TF for largest degrees}, we also have that the stalk complexes $H^0(X)[0]$ and $(1:t_0)(H^{-1}(X))[1]$ are in $\mathcal{H}_\phi$. We will freely use these facts throughout the rest of the proof.   

Let now $(X_\lambda )_{\lambda\in\Lambda}$ be a direct system in $\Ht$.
 Proving that the canonical map $\varinjlim H^k(X_\lambda )\longrightarrow H^k(\varinjlim_{\mathcal{H}_\phi}X_\lambda )$ is an isomorphism, for all $k\in\mathbb{Z}$, will be done in  several steps.

\vspace*{0.3cm}

\emph{Step 1: The result is true when $X_\lambda\in \D^{\leq -2}(R)$,
for all $\lambda\in\Lambda$:}

By Remark \ref{rem. truncation of phi}, we know that $(X_{\lambda})$
is a direct system of $\Ht \cap \mathcal{H}_{\phi_{\leq -1}}$ and
from Lemma \ref{lem. comparate sp-filtration} we get a triangle of
the form:
\begin{center}
$\xymatrix{T[1] \ar[r] & \varinjlim_{\mathcal{H}_{\phi_{\leq -1}}} X_{\lambda} \ar[r] & \limite{X_{\lambda}} \ar[r]^{\hspace{0.6 cm}+}  &},$
\end{center}
with $T\in\T_0$.

Since  $L(\phi_{\leq -1})\leq n$, the induction hypothesis gives
that the canonical map $\varinjlim{H^{k}(X_{\lambda})} \flecha
H^{k}(\varinjlim_{\mathcal{H}_{\phi_{\leq -1}}} X_{\lambda})$ is an
isomorphism, for all integers $k$. Note that this implies, in particular, that $H^{k}(\varinjlim_{\mathcal{H}_{\phi_{\leq -1}}} X_{\lambda})=0$, for $k\geq -1$.  From the
sequence of homologies applied to the previous triangle we then get that
the canonical maps
$\varinjlim{H^{k}(X_{\lambda})}\iso
H^{k}(\varinjlim_{\mathcal{H}_{\phi_{\leq -1}}}
X_{\lambda})\iso H^{k}(\varinjlim_{\Ht}
X_{\lambda})$ are isomorphisms, for $k\leq -3$  and  all $k\geq -1$.
By Lemma \ref{lem. domino effect}, we conclude that the canonical
morphism
 $\varinjlim{H^{-2}(X_{\lambda})}\cong H^{-2}(\varinjlim_{\Ht} X_{\lambda})$ is also an isomorphism. \\

\emph{Step 2: The result is true when $X_\lambda\in \D^{\leq -1}(R)$,
for all $\lambda\in\Lambda$:}

Using the octahedral axiom, for each $\lambda$ we obtain the following
commutative diagram:

$$
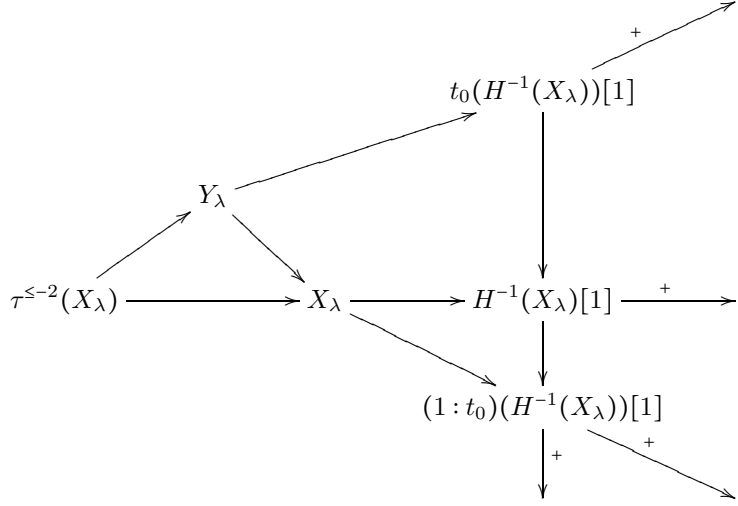
\xymatrix{ &&& &\\&&& t_{0}(H^{-1}(X_{\lambda}))[1] \ar[dd] \ar[ur]^{+} \\ & Y_{\lambda} \ar[dr] \ar[rru] && \\ \tau^{\leq -2}(X_{\lambda}) \ar[rr] \ar[ru] && X_{\lambda} \ar[rd] \ar[r] & H^{-1}(X_{\lambda})[1] \ar[r]^{\hspace{0.6 cm}+} \ar[d] & \\ &&& (1:t_{0})(H^{-1}(X_{\lambda}))[1] \ar[d]^{+} \ar[rd]^{+} \\ &&&&&}$$

From the triangle $\tau^{\leq -2}(X_\lambda) \flecha
Y_\lambda \flecha t_0(H^{-1}(X_{\lambda}))[1] \xymatrix{\ar[r]^{+} & }$ we get
that $Y_\lambda\in\U$, while from the triangle
$(1:t_0)(H^{-1}(X_{\lambda})) \flecha
Y_\lambda \flecha X_\lambda \xymatrix{\ar[r]^{+} & }$ and
the fact that $(1:t_0)(H^{-1}(X_{\lambda}))\in\T\F_{-1}$ (see
Example \ref{exem.TF for largest degrees}) we get that
$Y_\lambda\in\U^\perp [1]$. We then have that
 $Y_{\lambda} \in \Ht$, for each $\lambda\in\Lambda$. Moreover, we get direct systems of exact sequences in $\Ht$:

$$ \xymatrix{0 \ar[r] & t_0(H^{-1}(X_\lambda))[0] \ar[r] & \tau^{\leq -2}(X_\lambda) \ar[r] &  Y_\lambda \ar[r] & 0  & \ar@{}[d]^{(\ast)}\\
0\ar[r]  & Y_\lambda \ar[r] &  X_\lambda \ar[r] &  (1:t_0)(H^{-1}(X_\lambda ))[1] \ar[r] & 0 & }$$

Applying the direct limit functor and putting $T:=t_0(\varinjlim
H^{-1}(X_\lambda ))\cong\varinjlim t_0(H^{-1}(X_\lambda ))$, we get
an exact sequence in $\Ht$
\begin{center}
$ T[0]\flecha \varinjlim_{\Ht}(\tau^{\leq -2}(X_\lambda)) \xymatrix{\ar[r]^{p} &}\varinjlim_{\Ht}(Y_\lambda)\flecha 
0.$
\end{center} 
If $W:=\text{Ker}_{\Ht}(p)$ then, by Lemma \ref{lem. T[0]
closed for subobjects}, we have an exact sequence $0 \flecha
T'[0]\flecha T[0]\flecha W\flecha 0$ in $\Ht$,
with $T'\in\T_0$. This implies that $W\in \D^{[-1,0]}(R)$ and that
$H^k(W)\in\T_0$, for $k=-1,0$. But assertion 3 of Lemma \ref{lem.
last homology} then says that $H^{-1}(W)\in\F_0$, and this implies
that $W\cong H^0(W)[0]\in\T_0[0]$. We then have isomorphisms
$H^k(X_\lambda )=H^k(\tau^{\leq -2}(X_\lambda))\cong H^k(Y_\lambda )$
and $H^k(\varinjlim_{\Ht}(\tau^{\leq -2}(X_\lambda) ))\cong
H^k(\varinjlim_{\Ht}(Y_\lambda ))$, for all $k\leq -2$. By step 1 of
this proof, we get an isomorphism
\begin{center}
$\varinjlim (H^k(Y_\lambda ))\cong\varinjlim (H^k(X_\lambda ))\iso H^k(\limite (\tau^{\leq -2}(X_\lambda )))\cong H^k(\limite (Y_\lambda)), $
\end{center}
for all $k\leq -2$. Taking into account that $Y_\lambda\in\Ht\cap\mathcal D^{\leq -1}(R)$, for each $\lambda\in\Lambda$,   Lemma \ref{lem. domino effect} tells us that  the canonical map $\varinjlim (H^k(Y_\lambda
))\longrightarrow H^k(\limite Y_\lambda)$ is an isomorphism, for
all $k\in\mathbb{Z}$. \\

Bearing in mind that $L(\phi_{\leq -1})\leq n$,  Lemma \ref{lem. comparate sp-filtration} and the
 induction hypothesis imply that we have isomorphisms
 $\varinjlim H^k(X_\lambda )\cong H^k(\varinjlim_{\mathcal{H}_{\phi_{\leq -1}}}X_\lambda)\cong H^k(\limite
 X_\lambda)$, for all $k\leq -3$. They also imply the following facts:

 \begin{enumerate}
 \item $\varinjlim H^{-2}(X_\lambda )\cong  H^{-2}(\varinjlim_{\mathcal{H}_{\phi_{\leq -1}}}X_\lambda)\flecha H^{-2}(\limite X_\lambda
 )$ is a monomorphism with cokernel in $\T_0$;
\item $\varinjlim H^{-1}(X_\lambda )\cong  H^{-1}(\varinjlim_{\mathcal{H}_{\phi_{\leq -1}}}X_\lambda)\flecha H^{-1}(\limite X_\lambda
 )$ is an epimorphism with kernel in $\T_0$.
 \end{enumerate}

We next apply the direct limit functor to the second direct system
in (*) to get an exact sequence in $\Ht$:

$$\xymatrix{\limite{Y_{\lambda}} \ar[r]^{\alpha} & \limite{X_{\lambda}} \ar[r] & \varinjlim{(1:t_{0})(H^{-1}(X_{\lambda}))}[1] \ar[r] & 0}$$

We put $Z:=\Imagen(\alpha)$ and, using Proposition \ref{prop. stalk
heart} and Example \ref{exem.TF for largest degrees},   consider the
two induced short exact sequences in $\Ht$:

$$\xymatrix{0 \ar[r] & W^{'} \ar[r] & \limite{Y_{\lambda}} \ar[r] & Z \ar[r] & 0 \ \ \text{ and } \ \    0 \ar[r] & Z \ar[r] & \limite{X_{\lambda}} \ar[r] & \varinjlim{(1:t_0)(H^{-1}(X_{\lambda}))}[1] \ar[r] & 0}$$
From the first one we get that $Z\in \D^{\leq -1}(R)$ and, using the
exact sequence of homologies for the second one, we get that
$H^k(Z)\cong H^k(\limite X_\lambda)$, for all $k\leq -2$. On the
other hand, when applying the exact sequence of homologies to the
triangle $\tau^{\leq -2}(X_\lambda) \flecha
Y_\lambda \flecha t_0(H^{-1}(X_\lambda
))[1]\xymatrix{\ar[r]^{+} & }$, we get that
 the canonical map $H^k(Y_\lambda
)\flecha H^k(X_\lambda)$ is an isomorphism, for all
$\lambda\in\Lambda$ and all $k\leq -2$. If $q:\limite
Y_\lambda \flecha Z$ is the morphism in the first sequence
above, then we get that $\varinjlim (H^k(Y_\lambda ))\cong
H^k(\limite Y_\lambda
)\xymatrix{\ar[r]^{H^k(q)} & }H^k(Z)\cong\varinjlim
H^k(X_\lambda )$ is an isomorphism, for all $k\leq -3$ and a
monomorphism with cokernel in $\T_0$, for $k=-2$. This shows that
$W'\in \D^{[-1,0]}(R)$. Moreover, we have an induced exact sequence

$$\xymatrix{0 \ar[r] & \varinjlim H^{-2}(Y_\lambda ) \ar[r]^{\hspace{0.4 cm}H^{-2}(q)} & H^{-2}(Z) \ar[r] &  H^{-1}(W')\ar[r] & \varinjlim H^{-1}(Y_\lambda )}$$
Since $H^{-1}(Y_\lambda )=t_0(H^{-1}(X_\lambda ))$ and
$\text{Coker}(H^{-2}(q))$ are in $\T_0$ we derive that
$H^{-1}(W')\in\T_0$. But assertion 3 of Lemma \ref{lem. last
homology} says that $H^{-1}(W')\in\F_0$. It follows that
$H^{-1}(W')=0$, which implies that the canonical map
\begin{center}
$\varinjlim H^{-2}(X_\lambda)\cong\varinjlim H^{-2}(Y_\lambda)\cong H^{-2}(\limite Y_\lambda )\xymatrix{\ar[r]^{H^{-2}(q)} & } H^{-2}(Z)\cong H^{-2}(\limite X_\lambda)
$
\end{center} 
is an isomorphism.  We then get that the canonical map $\varinjlim H^{k}(X_\lambda)\flecha H^k(\limite
X_\lambda)$ is an isomorphism, for all $k\leq -2$. By Lemma
\ref{lem. domino effect}, we conclude that this morphism is also an
isomorphism, for all $k\in\mathbb{Z}$.\\

\emph{Step 3: The general case:}

Let $(X_{\lambda})$ be an arbitrary direct system of $\Ht$. We get
the following commutative diagram, with exact rows, in $\Ht$:

$$\xymatrix{ & \limite \tau^{\leq -1}(X_{\lambda}) \ar[r] \ar[d]_{\epsilon}& \limite X_\lambda \ar[r] \ar@{=}[d] & \limite(H^{0}(X_{\lambda})[0]) \ar[d] \ar[r] & 0 \\ 0 \ar[r] & \tau^{\leq -1}(\limite X_{\lambda}) \ar[r] & \limite X_{\lambda} \ar[r] & H^{0}(\limite {X_{\lambda}})[0] \ar[r] & 0 }$$

Its right vertical arrow is an isomoprhism since
$H^0:\Ht \flecha R \Mod$ preserves direct limits. By ker-coker
lemma, the morphism $\epsilon$ is  an epimorphism. We next consider
the exact sequence $0\flecha W''\flecha \linebreak \limite\tau^{\leq
-1}X_\lambda\xymatrix{\ar[r]^{\epsilon} & }\tau^{\leq
-1}(\limite X_\lambda )\flecha 0$ in $\Ht$. When localizing at a
prime $\mathbf{p}\in\Spec(R)\setminus\phi (0)$, the filtration
$\phi_\mathbf{p}$ has length $\leq n$. By the induction hypothesis,
we then get isomorphisms \linebreak $\varinjlim
H^k(R_\mathbf{p}\otimes_R\tau^{\leq
-1}X_\lambda)\iso H^k(\limite
(R_\mathbf{p}\otimes_R\tau^{\leq -1}X_\lambda ))$ and $\varinjlim
H^k(R_\mathbf{p}\otimes_RX_\lambda) \iso H^k(\limite
(R_\mathbf{p}\otimes_RX_\lambda ))$, for all $k\in\mathbb{Z}$. But
the canonical morphism $\tau^{\leq
-1}(R_\mathbf{p}\otimes_RX_\lambda)=R_\mathbf{p}\otimes_R\tau^{\leq
-1}X_\lambda\flecha R_\mathbf{p}\otimes_RX_\lambda$ is an
isomorphism  in $D(R_\mathbf{p})$ since
$R_\mathbf{p}\otimes_RH^0(X_\lambda )=0$, for all
$\lambda\in\Lambda$. It follows that
$H^k(1_{R_\mathbf{p}}\otimes\epsilon ):H^k(\limite
(R_\mathbf{p}\otimes_R\tau^{\leq -1}X_\lambda ))\flecha
H^k(\limite (R_\mathbf{p}\otimes_RX_\lambda ))$ is an isomorphism,
for all $k\in\mathbb{Z}$, which implies that
$1_{R\mathbf{p}}\otimes\epsilon$ is an isomorphism in
$D(R_\mathbf{p})$, for all $\mathbf{p}\in\Spec(R)\setminus\phi (0)$.
This in turn implies that $\text{Supp}(H^k(W''))\subseteq\phi (0)$
(equivalently $H^k(W'')\in\T_0$), for all $k\in\mathbb{Z}$. By a
reiterated use of assertion 3 of Lemma \ref{lem. last homology}, we
conclude that $W''\cong T[0]$, for some $T\in\T_0$.\\

By applying  the exact sequence of homologies to the triangle
\begin{center}
$T[0]\flecha \limite\tau^{\leq
-1}X_\lambda \xymatrix{\ar[r]^{\epsilon} & }\tau^{\leq
-1}(\limite X_\lambda )\xymatrix{\ar[r]^{+} & }$
\end{center}
and by using
step 2 of this proof, we get  isomorphisms $\varinjlim H^k(X_\lambda
)\cong H^k(\limite\tau^{\leq -1}X_\lambda
)\iso H^k(\tau^{\leq -1}(\limite
X_\lambda )) \linebreak =H^k(\limite X_\lambda )$, for all $k\leq -2$, and an
exact sequence

\begin{center}
$0\flecha \varinjlim H^{-1}(X_\lambda)\flecha
H^{-1}(\limite X_\lambda )\flecha T\flecha 0.$
\end{center} By
Proposition \ref{prop.gathering properties for bounded}, we conclude
that the canonical map $\varphi_k:\varinjlim H^k(X_\lambda )\flecha
H^k(\limite X_\lambda )$ is an isomorphism,  for all $k\leq -1$, which, by Lemma \ref{lem. domino effect}, means that $\varphi_k$ is an isomorphism for all $k\in\mathbb{Z}$.
\end{proof}

\subsection{The infinite case}

We are ready to prove the first main result of the paper.

\begin{teor}\label{teo. H is a Grothendieck c.}
Let $R$ be a commutative Noetherian ring and let $\phi$ be a left bounded sp-filtation of its spectrum. Then the heart $\Ht$ is a
Grothendieck category.
\end{teor}
\begin{proof}
 Without loss of
generality, we may and shall assume  that $\phi (j)\neq\emptyset$,
for all $j\in\mathbb{Z}$ and that $\phi$ is of the form $\cdots =\phi
(-n-k)=\cdots=\phi (-n-1)=\phi (-n)\supsetneq \phi(-n+1)\cdots \supseteq\phi
(0)\supseteq\phi (1)\supseteq \cdots$, for some $n\in\mathbb{N}$.\\

By Proposition \ref{generator of H} and by \cite[Proposition
3.4]{PS}, it is enough to check that, for each direct system
$(X_\lambda )_{\lambda\in\Lambda}$ in $\Ht$ and each
$j\in\mathbb{Z}$, the canonical map $\varphi_j:\varinjlim
H^j(X_\lambda )\flecha H^j(\limite X_\lambda )$ is an
isomorphism. Moreover, by Lemma \ref{lem. bounded heart}, it is
enough to do this for $j\geq -n$. Note that if $\mathbf{p}\in\Spec
(R)\setminus (\bigcap_{i\in\mathbb{Z}}\phi (i))$, then the
associated sp-filtration $\phi_\mathbf{p}$ of $\Spec(R_\mathbf{p})$
is either trivial (i.e. all its members are empty) or eventually trivial. By Theorem \ref{teo. AB5
finite},  we conclude that
$1_{R_\mathbf{p}}\otimes\varphi_j:R_\mathbf{p}\otimes_R(\varinjlim
H^j(X_\lambda ))\flecha R_\mathbf{p}\otimes_RH^j(\limite
X_\lambda )$ is an isomorphism, for all $j\in\mathbb{Z}$. We then
get that $\Supp(\Ker(\varphi_j))\cup\Supp(\Coker
(\varphi_j))\subseteq\bigcap_{i\in\mathbb{Z}}\phi (i)$, for all
$j\geq -n$.\\

Given such a direct system $(X_\lambda )$, we obtain a direct system
$(\tau^{\leq j}(X_\lambda ))_{\lambda\in\Lambda}$ in $\D(R)$, for each
$j\in\mathbb{Z}$, together with a morphism of ($\Lambda$-)direct
systems $(\tau^{\leq j}(X_\lambda) )\flecha \tau^{\leq j}(L)$,
where $L:=\limite X_\lambda$. We will prove at once, by induction on
$k\geq 0$,  the following two facts and the proof will be then
finished:

\begin{enumerate}
\item The canonical map $\varinjlim H^{-n+k}(X_\lambda
)\flecha H^{-n+k}(L)$ is an isomorphism;
\item The induced map $\varinjlim{\Hom_{D(R)}(R/\mathbf{p}[-j],\tau^{\leq
-n+k}(X_{\lambda})[s])}\flecha
\Hom_{\D(R)}(R/\mathbf{p}[-j],\tau^{\leq -n+k}(L) [s])$ is an isomorphism,
for all $k,s\in\mathbb{Z}$ and all $\mathbf{p}\in\Spec(R)$.
\end{enumerate}

For $k=0$, Proposition \ref{prop.gathering properties for bounded}
proves fact 1. On the other hand, by Lemma \ref{lem. bounded heart},
we have  isomorphisms $\tau^{\leq -n}(X_\lambda)\cong H^{-n}(X_\lambda
)[n]$ and $\tau^{\leq -n}(L)\cong H^{-n}(L)[n]$ in $\D(R)$. Proving
fact 2 is equivalent in this case to proving that the canonical map

$$\varinjlim\Ext_R^{j+s}(R/\mathbf{p},H^{-n}(X_\lambda ))\flecha \Ext_R^{j+s}(R/\mathbf{p},H^{-n}(L))$$
is an isomorphism. This follows immediately from fact 1 for $k=0$
since the functor $\Ext_R^{j+s}(R/\mathbf{p},?)$ preserves direct
limits.\\

Let $k\geq 0$ be a natural number and suppose now that facts 1 and 2
are true for the natural numbers $\leq k$. Let $Z\in\Ht$ be any
object. By Remark \ref{rem. truncation of phi}, we have that
$\tau^{\leq -n+k+1}(Z)\in \mathcal{H}_{\phi_{\leq
-n+k+3}}\subseteq\mathcal{U}_{\phi_{\leq -n+k+3}}^\perp [1]$. We
next consider the following triangle in $\D(R)$:
$$\xymatrix{\tau^{\leq -n+k+1}(Z)[-1] \ar[r] & H^{-n+k+1}(Z)[n-k-2] \ar[r] & \tau^{\leq -n+k}(Z) \ar[r] & \tau^{\leq -n+k+1}(Z)}$$
If $\mathbf{p}\in \phi(-n+k+3)$ is any prime, then the functor
$\Hom_{\D(R)}(R/\mathbf{p}[n-k-2],?)$  vanishes on
$\tau^{\leq -n+k+1}(Z)[-1]$ and $\tau^{\leq -n+k+1}(Z)$ while
$\Hom_{\D(R)}(R/\mathbf{p}[n-k-3],?)$ vanishes on
$\tau^{\leq -n+k+1}(Z)[-1]$. We then get:

\begin{enumerate}
\item[1')] The morphism $\Hom_{R}(R/\mathbf{p},
H^{-n+k+1}(Z))\cong\Hom_{\D(R)}(R/\mathbf{p}[n-k-2],H^{-n+k+1}(Z)[n-k-2])\flecha$
\\ $\Hom_{\D(R)}(R/\mathbf{p}[n-k-2], \tau^{\leq -n+k}(Z))$ is an
isomorphism.
\item[2')] The morphism $\Ext^{1}_{R}(R/\mathbf{p},H^{-n+k+1}(Z))\cong\Hom_{D(R)}(R/\mathbf{p}[n-k-3],H^{-n+k+1}(Z)[n-k-2])
\flecha$ \\ $\Hom_{\D(R)}(R/\mathbf{p}[n-k-3], \tau^{\leq
-n+k}(Z))\cong\Hom_{\D(R)}(R/\mathbf{p}[n-k-2], \tau^{\leq
-n+k}(Z)[1])$ is a monomorphism.
\end{enumerate}
We now get the following chain of isomorphisms, where we write over
each arrow the reason for it to be an isomorphism:

$$\Hom_R(R/\mathbf{p},\varinjlim H^{-n+k+1}(X_\lambda ))\stackrel{can}{\longleftarrow}\varinjlim\Hom_R(R/\mathbf{p},H^{-n+k+1}(X_\lambda ))
\stackrel{fact
1'}{\longrightarrow}\varinjlim\Hom_{\D(R)}(R/\mathbf{p}[n-k-2],\tau^{\leq -n+k}(X_\lambda))$$

$$\stackrel{induction}{\longrightarrow}
\Hom_{\D(R)}(R/\mathbf{p}[n-k-2],\tau^{\leq -n+k}(L))\stackrel{fact
1'}{\longleftarrow}\Hom_{R}(R/\mathbf{p},H^{-n+k+1}(L)).
$$
It is not difficult to see that the resulting isomorphism is the one
obtained by applying the functor $\Hom_R(R/\mathbf{p},?)$ to the
canonical morphism $\varphi_{-n+k+1}:\varinjlim H^{-n+k+1}(X_\lambda
)\longrightarrow H^{-n+k+1}(L)$. If follows that \linebreak
$\Hom_R(R/\mathbf{p},\Ker (\varphi_{-n+k+1}))=0$, for all
$\mathbf{p}\in\phi (-n+k+3)$, and hence (see the second   paragraph
of this proof) the same is true for all
$\mathbf{p}\in\Supp(\Ker(\varphi_{-n+k+1}))$. Therefore
$\varphi_{-n+k+1}$ is a monomorphism. \\

We claim that the  morphism
$\Ext_R^1(R/\mathbf{p},\varphi_{-n+k+1}):\Ext_R^1(R/\mathbf{p},\varinjlim
H^{-n+k+1}(X_\lambda
))\flecha \Ext_R^1(R/\mathbf{p},H^{-n+k+1}(L))$ is a
monomorphism, for all $\mathbf{p}\in\phi (-n+k+3)$. By taking then
the sequence of $\Ext(R/\mathbf{p},?)$ associated to the exact
sequence $\xymatrix{0 \ar[r] & \varinjlim H^{-n+k+1}(X_\lambda
) \ar[rr]^{\hspace{0.4 cm}\varphi_{-n+k+1}} && H^{-n+k+1}(L) \ar[r] & \Coker
(\varphi_{-n+k+1})\ar[r] & 0}$, we will derive that
$\Hom_R(R/\mathbf{p},\Coker(\varphi_{-n+k+1}))=0$, for all
$\mathbf{p}\in\phi (-n+k+3)$. Then,  arguing as in the case of the
kernel, we will deduce that $\Coker (\varphi_{-n+k+1})=0$, so that
the $k+1$-induction step for fact 1 above will be covered. In order
to settle our claim, we consider the commutative diagram

$$\xymatrix{\Ext^{1}_{R}(R/\mathbf{p},\varinjlim{H^{-n+k+1}(X_{\lambda})})  & & \\ \varinjlim{\Ext^{1}_{R}(R/\mathbf{p},H^{-n+k+1}(X_{\lambda}))} \ar[rr] \ar[u]_{\wr} \ar[d] && \Ext^{1}_{R}(R/\mathbf{p},H^{-n+k+1}(L)) \ar[d] \\ \varinjlim{\Hom_{\D(R)}(R/\mathbf{p}[n-k-2],\tau^{\leq -n+k}(X_{\lambda})[1])} \ar[rr]^{\sim} && \Hom_{\D(R)}(R/\mathbf{p}[n-k-2],\tau^{\leq -n+k}(L)[1])}$$

The induction hypothesis for fact 2 gives that the lower horizontal
arrow is an isomorphism. On the other hand, by 2') above, the
downward left vertical arrow is a monomorphism. Then our claim
follows immediately. On the other hand, the $k+1$-step of the induction for fact 2  follows from Lemma \ref{lem.domino-effect2}.
\end{proof}

\begin{rem}
Theorem A in the introduction is just a geometric translation of Theorem \ref{teo. H is a Grothendieck c.}.  
\end{rem}

\begin{exem}
If $R$ is a commutative Noetherian ring such that $\Spec(R_\p)$ is a finite set, for all $\p \in \Spec(R)$ (e.g.  if $R$ has Krull dimension $\leq$ 1), then the heart of every compactly generated t-structure in $\D(R)$ is a Grothendieck category.
\end{exem}
\begin{proof}
Let $(\mathcal{U},\mathcal{U}^{\perp}[1])$ be a compactly generated t-structure in $\D(R)$ and let $\phi$ be the associated sp-filtration. It is clear that $\phi_{\p}$ is a left bounded sp-filtration (since $\Spec(R_\p)$ is finite), for all $\p\in \Spec(R)$. Using the previous theorem, we get that $\mathcal{H}_{\phi_\p}$ is a Grothendieck category, in particular an AB5 abelian category, for all $\p\in \Spec(R)$. Now apply Corollary \ref{cor. Hp is AB5} and Proposition \ref{generator of H}.
\end{proof}
\section{The module case}
 
The goal of this section is to identify all the compactly generated t-structures in $\D(R)$  whose heart is a module category. We start by giving some comments on t-structures in a finite product of triangulated categories.

\subsection{t-structures in a finite direct product of triangulated categories}

Let $\D_1,\dots,\D_n$ be triangulated categories with arbitrary (set-indexed) coproducts. The product category $\D =\D_1\times\dots \times \D_n$ has an obvious structure of triangulated category with the relevant concepts defined componentwise. Then one can view each $\D_i$ as a full triangulated subcategory of $\D$ by identifying each object $X$ of $\D_i$ with the object $(0,\dots,0,X,0,\dots,0)$ of $\D$, with $X$ in the $i$-th position. With this convention, we have the following useful result, where $\text{aisl}(\mathcal D)$ denote the (ordered by inclusion) class of aisles in the triangulated category $\mathcal D$.

\begin{lemma} \label{lem. product of aisles}
With the terminology of the previous paragraph, the following assertions hold:

\begin{enumerate}
\item If $\mathcal{U}$ is the aisle of a t-structure in $\D$, then $\mathcal{U}\cap\D_i$ is the aisle of a t-structure in $\D_i$, for $i=1,\dots,n$.
\item If $\mathcal{U}_i$ is the aisle of a t-structure in $\D_i$, for $i=1,\dots,n$, then the full subcategory $\mathcal{U}$ of $\D$ consisting of the $n$-tuples $(X_1,\dots,X_n)$ such that $X_i\in\mathcal{U}_i$, for all $i=1,\dots,n$, is the aisle of a t-structure in $\D$. We denote this aisle by $\oplus_{1\leq i\leq n}\hspace{0.04 cm}\mathcal{U}_i$.
\item The assignment $(\mathcal{U}_1\times \dots \times\mathcal{U}_n)\rightsquigarrow\oplus_{1\leq i\leq n}\hspace{0.04cm}\mathcal{U}_i$ defines an isomorphism of ordered classes $$\text{aisl}(\D_1)\times \dots \times\text{aisl}(\D_n) \iso \text{aisl}(\D).$$
\item If $\mathcal{H}_i$ denotes the heart of the t-structure $(\mathcal{U}_i,\mathcal{U}_i^\perp [1])$, for $i=1,\dots,n$, and $\mathcal{H}$ is the heart of the t-structure in $\D$ whose aisle is $\oplus_{1\leq i\leq n}\hspace{0.04 cm}\mathcal{U}_i$, then $\mathcal{H}$ is equivalent to $\mathcal{H}_1\times \dots \times\mathcal{H}_n$.
\end{enumerate}
\end{lemma}
\begin{proof}
The proof of ssertions 1, 2 and 3 is easy and left to the reader. As for assertion 4, put $\mathcal{U}:=\oplus_{1\leq i\leq n}\hspace{0.04cm}\mathcal{U}_i$. We then have $\mathcal{U}^\perp =\bigcap_{1\leq i\leq n} \hspace{0.04 cm}\mathcal{U}_i^\perp$, where $(-)^\perp$ denotes orthogonality in $\D$ and we are viewing $\mathcal{U}_i$ as a full subcategory of $\D$. Note that we then have  $\mathcal{U}_i^\perp=\D_1\times \dots \times \D_{i-1}\times\mathcal{U}_i^\circ\times\D_{i+1}\times \dots \times\D_n$, where $\mathcal{U}_i^\circ$ denotes the right orthogonal of $\mathcal{U}_i$ within $\D_i$. It then follows that $\mathcal{U}^\perp=\mathcal{U}_1^\circ \times \dots \times \mathcal{U}_n^\circ$, which implies that $\mathcal{H}=\mathcal{U}\cap \mathcal{U}^\perp [1]$ is equal to $\mathcal{H}_1\times \dots \times\mathcal{H}_n$.  Assertion 4 follows immediately from this. 
\end{proof}

\begin{exem}
Let $R$ be a not necessarily commutative ring and let $\{e_1,\dots,e_n\}$ be a family of nonzero orthogonal central idempotents of $R$ such that $1=\sum_{1\leq i\leq n}e_i$. If $R_i=Re_i$ and $X_i$ is a complex of $R_i$-modules, for all $i=1,\dots,n$, then $X=\oplus_{1\leq i\leq n}X_i$ is a complex of $R$-modules. Moreover, the assignment $(X_1,\dots,X_n)\rightsquigarrow\oplus_{1\leq i\leq n}X_i$ gives an equivalence of triangulated categories $\D(R_1)\times \dots \times\D(R_n) \iso \D(R)$. The previous lemma then says that all aisles of t-structures in $\D(R)$ are 'direct sums' of aisles of t-structures in the $\D(R_i)$, with the obvious meaning. 
\end{exem}

\subsection{The eventually trivial case for connected rings}

Let $e_1,\dots,e_n$ be the connected components of $R$.  The last lemma and example say that if $(\mathcal{U},\mathcal{U}^\perp [1])$ is a t-structure in $\D(R)$, then its heart is a module category if, and only if, the heart of the t-structure $(\mathcal{U}\cap\D(R_i),(\mathcal{U}^\perp\cap\D(R_i))[1])$ in $\D(R_i)$ is a module category, where $R_i=Re_i$,  for each $i=1,\dots,n$. As a consequence, our problem in the section reduces to the case when $R$ is connected. In this subsection we assume that $R$ is such and, as a first step towards the complete solution of the problem, we fix an eventually trivial sp-filtration $\phi$ of $\Spec(R)$ (see Subsection \ref{sub. compactly generated}).

\begin{prop} \label{prop.modular heart connected case}
Let us assume that $R$ is connected and let $\phi$ be an eventually trivial sp-filtration of $\Spec(R)$ such that $\phi (i)\neq\emptyset$, for some $i\in\mathbb{Z}$. The following assertions are equivalent:

\begin{enumerate}
\item  $\Ht$ is a module category;
\item There is an $m\in\mathbb{Z}$ such that $\phi (m)=\Spec(R)$ and $\phi (m+1)=\emptyset$;
\item There is an $m\in \mathbb{Z}$ such that the t-structure associated to $\phi$ is $(\D^{\leq m}(R),\D^{\geq m}(R))$. 
\end{enumerate}
In that case $\Ht$ is equivalent to $R-\text{Mod}$. 
\end{prop}
\begin{proof}
The implications $2)\Longleftrightarrow 3)\Longrightarrow 1)$ are clear and the final statement is a consequence of assertion 3.

$1)\Longrightarrow 2)$  Without loss of generality, we may and shall assume that  $\phi (0)\neq\emptyset =\phi (1)$. We then fix a progenerator $P$ of $\mathcal{H}=\Ht$ and will prove in several steps that $\T_0=R-\text{Mod}$. \\

\emph{Step 1: $H^0(P)$ is a finitely generated $R$-module}: We have   that $P\in \D^{\leq 0}(R)$ since $\Supp(H^i(P))\subseteq\phi (i)=\emptyset$, for all $i>0$. Then, by  Remark \ref{rem. truncation of phi}, we know that $\tau^{\leq -1}(P)\in \mathcal{H}_{\phi \leq 0}=\Ht$, and hence, we have an exact sequence in $\Ht$
$$\xymatrix{0 \ar[r] & \tau^{\leq -1}(P) \ar[r] & P \ar[r] & H^{0}(P)[0] \ar[r] & 0 & (\ast)}$$  
But, for each $R$-module $Y$, we have isomorphisms $\Hom_{R}(H^{0}(P),Y) \cong \Hom_{\D(R)}(H^{0}(P)[0],Y[0]) \cong \Hom_{\D(R)}(P,Y[0]).$ On the other hand, if $(T_\lambda)$ is a direct system in $\T_{0}$, then $\limite{T_\lambda[0]}\cong (\varinjlim{T_\lambda})[0]$ (see Proposition \ref{prop. stalk heart}). It follows that $\Hom_{R}(H^{0}(P),?)$ preserves direct limits of objects in $\T_0$, because $P$ is a finitely presented object of $\Ht.$ Using Remark \ref{remark class closed direct.limit},  we get the following chain of canonical isomorphisms, which proves that $H^0(P)$ is a finitely presented (=finitely generated) $R$-module. 

\begin{small}
$$\xymatrix{\varinjlim{\Hom_{R}(H^{0}(P),M_{\lambda})} & \varinjlim{\Hom_{R}(H^{0}(P),t_{0}(M_\lambda))} \ar[l]_{\sim} \ar[r]^{\sim \hspace{2.1 cm}} & \Hom_{R}(H^{0}(P),\varinjlim{t_{0}(M_{\lambda})})=\Hom_{R}(H^{0}(P),t_{0}(\varinjlim{M_{\lambda}})) \ar[d]^{\wr} \\ && \Hom_{R}(H^{0}(P),\varinjlim{M_\lambda})  }$$
\end{small}

\emph{Step 2: $H^0(P)$ is a progenerator of $\T_0$:} Step 1 proves that $H^0(P)$ is a compact object of $\T_0$. On the other hand, the functor $H^0:\mathcal{H}\longrightarrow R-\text{Mod}$ is right exact and its essential image is contained in $\T_0$. The facts that $\T_0[0]\subset\mathcal{H}$ and $P$ is a generator of $\mathcal{H}$ imply that $\T_0=\text{Gen}(H^0(P))$ and, hence, that $H^0(P)$ is a generator of $\T_0$. Finally, if $T\in\T_0$ and we apply
 $\Hom_{D(R)}(?,T[1])$ to the triangle associated to the sequence $(\ast)$, we obtain the following exact sequence in $R-\text{Mod}$:

$$\xymatrix{\Hom_{\D(R)}(\tau^{\leq -1}(P)[1],T[1])=0 \ar[r] & \Ext^{1}_{R}(H^{0}(P),T) \ar[r] & \Hom_{\D(R)}(P,T[1])=\Ext^{1}_{\Ht}(P,T[0])=0.}$$
This shows that $\Ext^{1}_{R}(H^{0}(P),T)=0$ and, hence, that $H^0(P)$ is a projective object of $\T_0$.\\

\emph{Step 3: $\T_0=R-\text{Mod}$:}  Using the bijection between hereditary torsion pairs and Gabriel topologies (see \cite[Chapter VI]{S}) and step 1, we know that there are ideals $\mathbf{a}_{1}, \dots, \mathbf{a}_{m}$ such that $R/\mathbf{a}_{i}\in \T_{0}$, for $i=1,\dots, m$, and there is an epimorphism $\underset{1 \leq i \leq m}{\oplus}\frac{R}{\mathbf{a}_i} \epic H^0(P)$. It follows that $\mathbf{a}=\underset{1 \leq i \leq m}{\cap} \mathbf{a}_{i} \subseteq \text{ann}_{R}(H^0(P))$ (commutativity of $R$ is essential here), where $\text{ann}_R(X)=\{a\in R:$ $aX=0\}$ for each $R$-module $X$. But $R/\mathbf{a}$ is also in $\T_{0}$, whence generated by $H^0(P)$. It follows that $\text{ann}_{R}(H^0(P)) \subseteq \mathbf{a}$ and, hence, that this inclusion is an equality. We then get that $\mathbf{a}=\text{ann}_{R}(H^0(P))$ is contained in any ideal $\mathbf{b}$ of $R$ such that $R/\mathbf{b}\in \T_{0}$, which also implies that $\T_0=\Gen(R/\mathbf{a})$. By \cite[Propositions VI.6.12 and VI.8.6]{S}, we know that $\mathbf{a}=Re$, for some idempotent element $e\in R$. The connectedness of $R$ gives that $\mathbf{a}=0$ or $\mathbf{a}=R$. But the second possibility is discarded because $\phi(0)\neq \emptyset$, and hence $\T_0\neq 0$. This shows that $\T_0=\Gen(R)=R-\text{Mod}$ or, equivalently, that $\phi (0)=\Spec(R)$.
\end{proof}

\subsection{Some auxiliary results}

\begin{prop} \label{prop.heart equivalent to quotient}
Let $Z$ be a sp-subset of $\Spec(R)$, and let $\phi$ be the sp-filtration given by $\phi(i)=\Spec(R)$, for all $i \leq 0$, and $\phi(i)=Z$, for all $i>0$. Then, for each $Y\in\Ht$, the $R$-module $H^{0}(Y)$ is $Z$-closed $R$-module. Moreover, the assignment $Y \rightsquigarrow H^{0}(Y)$ defines an equivalence of categories $H^{0}:\Ht \iso \frac{R\Mod}{\T_Z}$, whose inverse is $L(?[0])$.
\end{prop}
\begin{proof}
For each $Y\in \Ht$, we have a triangle of the form (see Lemma \ref{lem. bounded heart}):
$$\xymatrix{\tau^{>0}(Y)[-1] \ar[r] & H^{0}(Y)[0] \ar[r]^{\hspace{0.7 cm}\lambda_Y} & Y \ar[r]^{+} &  &(\ast) }$$
From Lemma  \ref{lem. last homology}, we get that $H^{0}(Y)$ is a $Z$-closed $R$-module. On the other hand, note that $H^{j}(\tau^{>0}(Y))\in \T_{Z}$, for all $j\in \mathbb{Z}$. Hence, $\tau^{>0}(Y)[k]\in \mathcal{U}_{\phi}$, for all $k\in \mathbb{Z}$. We then have $L(H^{0}(Y)[0])\cong Y$, since $Y\in \mathcal{U}^{\perp}[1]$.\\

Now, if $Y^{'}$ is an arbitrary object in $\Ht$, then applying the functor $\Hom_{\D(R)}(?,Y^{'})$ to the triangle $(\ast)$, we obtain that $\lambda_{Y}^{*}:=\Hom_{\D(R)}(\lambda_{Y},Y^{'}):\Hom_{\D(R)}(Y,Y^{'}) \flecha \Hom_{\D(R)}(H^{0}(Y)[0],Y^{'})$ is an isomorphism, since 
$\Hom_{\D(R)}(?,Y^{'})$ vanishes on $\tau^{>0}(Y)[k]$, for all $k\in \mathbb{Z}$. Similarly, applying the functor $\Hom_{\D(R)}(H^{0}(Y)[0],?)$ to the triangle $(\ast)$ associated to $Y^{'}$ and putting $(\lambda_{Y^{'}})_{*}=\Hom_{\D(R)}(H^{0}(Y)[0],\lambda_{Y^{'}})$, we obtain that $:\Hom_{\frac{R\Mod}{\T_Z}}(H^{0}(Y),H^{0}(Y^{'}))\cong\Hom_{R}(H^{0}(Y),H^{0}(Y^{'})) \cong\Hom_{\D(R)}(H^{0}(Y)[0],H^{0}(Y^{'})[0])\stackrel{(\lambda_{Y^{'}})_{*}}{\longrightarrow}\Hom_{\D(R)}(H^{0}(Y)[0],Y^{'})$ is an isomorphism. It is not difficult to see that the morphism \linebreak $(\lambda_{Y^{'}})^{-1}_{*} \circ \lambda_{Y}^{*}:\Hom_{\D(R)}(Y,Y^{'}) \iso \Hom_{\frac{R\Mod}{\T_Z}}(H^{0}(Y),H^{0}(Y^{'}))$ coincides with the morphism \linebreak $H^{0}: \Hom_{\D(R)}(Y,Y^{'}) \flecha \Hom_{\frac{R\Mod}{\T_Z}}(H^{0}(Y),H^{0}(Y^{'}))$. This shows that the functor $H^{0}:\Ht \flecha \frac{R\Mod}{\T_Z}$ is full and faithful. It is also dense. Indeed, if $F$ is a $Z$-closed $R$-module, then $F[0]\in \U$ and we have a triangle in $\D(R)$:
$$\xymatrix{\tau^{\leq}_{\phi}(F[-1])[1] \ar[r]^{\hspace{0.7 cm}f} & F[0] \ar[r]^{g \hspace{0.3 cm}} & L(F[0]) \ar[r]^{ \hspace{0.5 cm}+} &}$$  
Since $L(F[0])\in \Ht \subseteq \D^{\geq 0}(R)$, we know that $\tau^{\leq}_{\phi}(F[-1])[1] \in \D^{\geq0}(R)$ and that $H^{0}(\tau^{\leq}_{\phi}(F[-1])[1])=H^{1}(\tau^{\leq}_{\phi}(F[-1]))$, which is in $\T_{Z}$, is a submodule of $F$. It follows that $H^{0}(\tau^{\leq}_{\phi}(F[-1])[1])=0$. Thus, we have an exact sequence in $R\Mod$ of the form:
$$\xymatrix{0 \ar[r] & F \ar[r]^{H^{0}(g) \hspace{0.75 cm}} & H^{0}(L(F[0])) \ar[r] & H^{1}(\tau^{\leq}_{\phi}(F[-1])[1]) \ar[r] & 0}$$
Since $H^{1}(\tau^{\leq}_{\phi}(F[-1])[1])\in \T_{Z}$ and $F$ is a $Z$-closed $R$-module, we obtain that the previous sequence split. It follows $H^{1}(\tau^{\leq}_{\phi}(F[-1])[1])=0$ and $H^{0}(L(F[0]))\cong F$.  
\end{proof}

In the rest of the paper we will use the following terminology.

\begin{defi}
Let $S\subseteq W$ be subsets of $\Spec (R)$. We will say that $S$ is \emph{stable under specialization (resp. generalization) within $W$}, when the following property holds:

\vspace*{0.3cm}

($\dagger$) If $\mathbf{p}\subseteq\mathbf{q}$ are prime ideals in $W$ such that $\mathbf{p}\in S$ (resp. $\mathbf{q}\in S$), then $\mathbf{q}\in S$ (resp. $\mathbf{p}\in S$).

\vspace*{0.3cm}
 Note that in such case $S$ need not be stable under specialization (resp. generalization) in $\Spec( R)$. 
\end{defi}

\begin{lemma}\label{lem. product of quotient category}
Let $Z \subseteq \Spec(R)$ be  a sp-subset and suppose that $\Spec(R)\setminus Z=\tilde{V} \bigcupdot \tilde{W}$(disjoint union!),  where $\tilde{V}$ and $\tilde{W}$ are stable under specialization within $\Spec(R)\setminus Z$. If we put $V=\tilde{V}\cup Z$ and $W=\tilde{W} \cup Z$, then the following assertions holds:
\begin{enumerate}
\item[1)] $V$ and $W$ are sp-subsets of $\Spec(R)$;
\item[2)] The category $\frac{\text{R-Mod}}{\T_Z}$ is equivalent to $\frac{\T_V}{\T_Z} \times \frac{\T_W}{\T_Z}$;
\item[3)] We have canonical equivalences of categories $\frac{\mathcal T_V}{\mathcal T_Z}\stackrel{\cong}{\longrightarrow}\frac{R\Mod}{\mathcal T_W}$ and $\frac{\mathcal T_W}{\mathcal T_Z}\stackrel{\cong}{\longrightarrow}\frac{R\Mod}{\mathcal T_V}$. 
\end{enumerate}
\end{lemma}
\begin{proof}
All throughout the proof, for any sp-subset $Z$ of $\Spec (R)$ we shall identify $R\Mod /\mathcal T_Z$ with the associated Giraud subcategory $\mathcal{G}_Z$ of $R\Mod$ consisting of  the $Z$-closed $R$-modules (see Subsection \ref{sub. localization}). 

1) This assertion is straightforward. \\

2) Clearly $\tilde{\T_{V}}:=q(\T_V)=\frac{\T_V}{\T_Z}$ is a hereditary torsion class in $\frac{R\text{-Mod}}{\T_Z}$ which, as a full subcategory of $R\Mod$, gets identified with $\mathcal T_V\cap\mathcal{G}_Z$. A similar fact is true when replacing $V$ by $W$. 

We claim that $\Hom_{\mathcal{G}_Z}(T_W,T_V)=0=\Ext_{\mathcal{G}_Z}^1(T_W,T_V)=0$, for all $T_W\in\tilde{\T}_W$ and all $T_V\in\tilde{\T}_V$ and, once this is proved, the same will be true with the roles of $V$ and $W$ exchanged. Note that $T_V$ is a $Z$-torsionfree and $V$-torsion $R$-module. Since $\mathcal T_V$ is closed under taking injective envelopes, it follows that the minimal injective resolution of $T_V$ in $R\Mod$ has the form 

$$0\rightarrow T_V\longrightarrow I_{\tilde{V}}^0\longrightarrow  I_{\tilde{V}}^1\oplus I_Z^1\longrightarrow \cdots \longrightarrow  I_{\tilde{V}}^n\oplus I_Z^n\longrightarrow \cdots ,$$
with the convention that $I_{\tilde{V}}^k\in\text{Add}(\oplus_{\mathbf{p}\in\tilde{V}}E(R/\mathbf{p}))$ and $I_Z^k\in\text{Add}(\oplus_{\mathbf{p}\in Z}E(R/\mathbf{p}))$, for all $k\geq 0$. When applying the functor $q:R\Mod\longrightarrow R\Mod/\mathcal T_Z\cong\mathcal{G}_Z$, we get an  injective resolution 

$$0\rightarrow T_V\longrightarrow I_{\tilde{V}}^0\longrightarrow  I_{\tilde{V}}^1\longrightarrow \cdots \longrightarrow  I_{\tilde{V}}^n\longrightarrow \cdots $$
in $\mathcal{G}_Z$. All its terms are then injective $W$-torsionfree $R$-modules, which implies that $\Ext_{\mathcal{G}_Z}^k(T_W,T_V)=0$, for all $k\geq 0$, and our claim follows.

Let now $Y\in\mathcal{G}_Z$  be any object.  If $F\in\mathcal{G}_Z$ is such that $\Hom_{R\Mod/\mathcal T_Z}(q(T),F)=0$, for all $T\in\mathcal T_V$, the adjunction $(q,j)$ gives that $j(F)$ is a $V$-torsionfree $R$-module, so that its injective envelope $E(j(F))$ is in $\text{Add}(\oplus_{\mathbf{p}\in\Spec(R)\setminus V}E(R/\mathbf{p}))=\text{Add}(\oplus_{\mathbf{p}\in\tilde{W}}E(R/\mathbf{p}))$. But then $E(j(F))$  is in $\mathcal T_W$, which implies that $F\in\mathcal T_W\cap\mathcal{G}_Z=\tilde{\T_W}$. By the previous paragraph, we have that $\text{Ext}_{\mathcal{G}_Z}^1(F,?)$ vanishes on $\tilde{\T_V}$. It follows that if $Y\in\mathcal{G}_Z$ is any object and  $0\rightarrow\tilde{T}_V\longrightarrow Y\longrightarrow F\rightarrow 0$ is the canonical exact sequence in $\mathcal{G}_Z$ associated to torsion pair $(\tilde{\T_V},\tilde{\T_V}^\perp)$, then this sequence splits. Therefore $Y$  decomposes as $Y=\tilde{T}\oplus F$, where $\tilde{T}\in\tilde{\T_V}$ and $F\in\tilde{\T_W}$. Assertion 2 is now clear. \\

3) We shall prove that $\mathcal{T}_V\cap\mathcal{G}_Z=\mathcal{G}_W$. From the equivalences of categories $\frac{\mathcal T_V}{\mathcal T_Z}\cong\mathcal{T}_V\cap\mathcal{G}_Z$ and $\frac{R\Mod}{\mathcal T_W}\cong\mathcal{G}_W$ the first 'half' of assertion 3 will follow then automatically. The other 'half' is obtained by symmetry.

 It follows from the arguments in the proof of assertion 2 that if $\tilde{T}_V\in\tilde{\T_V}=\mathcal{T}_V\cap\mathcal{G}_Z$ then $\tilde{T}_V$ is $W$-torsionfree. Let us take now $T_W\in\mathcal{T}_W$ and consider any exact sequence in $R\Mod$

$$0\rightarrow\tilde{T}_V\stackrel{u}{\longrightarrow} M\longrightarrow T_W\rightarrow 0.$$
When applying the functor $q:R\Mod \longrightarrow R\Mod /\mathcal{T}_Z\cong\mathcal{G}_Z$, using the proof of assertion 2 we get a split exact sequence in $\mathcal{G}_Z$.  In particular, we have that $(j\circ q)(u)$ is a section in $R\Mod$. But, bearing in mind that $\tilde{T}_V$ is $Z$-closed, we can identify $(j\circ q)(u)$ with the composition $\tilde{T}_V\stackrel{u}{\longrightarrow} M\stackrel{\mu_M}{\longrightarrow} (j\circ q)(M)$. It follows that $u$ is a section in $R\Mod$, so that $\Ext_R^1(?,\tilde{T}_V)$ vanishes on $\mathcal{T}_W$. Therefore $\tilde{T}_V$ is in $\mathcal{G}_W$. 

Let now take  $Y\in\mathcal{G}_W$. Then $Y$ is $Z$-closed and assertion 2 gives a decomposition $Y=\tilde{T}_V\oplus\tilde{T}_W$, where $\tilde{T}_V\in\tilde{\T_V}$ and $\tilde{T}_W\in\tilde{\T_W}$. But $\tilde{T}_W$ is then $W$-torsion and $W$-closed, which implies that it is zero. Therefore we have $Y\cong\tilde{T}_V\in\mathcal{T}_V\cap\mathcal{G}_Z$. 
\end{proof}

\begin{lemma}\label{lem. q(U) is a t-struc.}
Let $(\mathcal{U},\mathcal{U}^\perp [1])$ be a compactly generated t-structure in $\mathcal D(R)$, let $Z\subseteq\Spec (R)$ be an sp-subset such that $\mathcal{U}_Z:=\{X\in\mathcal{D}(R):$ $\Supp (H^i(X))\subseteq Z\text{, for all }i\in\mathbb{Z}\}$ is contained in $\mathcal{U}$, and let $\mu :R\longrightarrow R_Z$  and $q:\mathcal D(R)\longrightarrow\mathcal{D}(\frac{R\Mod}{\mathcal T_Z})$ be the canonical ring homomorphism and quotient functor, respectively. The following assertions hold:

\begin{enumerate}
\item If $q':\mathcal{D}(R_Z)\longrightarrow\mathcal{D}(\frac{R\Mod}{\mathcal T_Z})$ is the triangulated functor given by the quotient functor $q':R_Z\Mod\longrightarrow\frac{R_Z\Mod}{\mathcal\mu_*^{-1}(T_Z)}\cong\frac{R\Mod}{\mathcal T_Z}$, then the following conditions are equivalent for a complex $M\in\mathcal D(R)$:

\begin{enumerate}
\item $M$ is in $\mathcal{U}$;
\item $q(M)$ is in $q(\mathcal{U})$;
\item $\mu_*(R_Z\otimes_R^\mathbf{L}M)$ is in $\mathcal{U}$;
\item $q'(R_Z\otimes_R^\mathbf{L}M)$ is in $q(\mathcal{U})$. 
\end{enumerate}
\item $(q(\mathcal{U}),q(\mathcal{U}^{\perp})[1])$ is a t-structure in $\D(\frac{R\Mod}{\T_Z})$ and its heart is equivalent to the heart of $(\mathcal{U},\mathcal{U}^\perp [1])$.
\end{enumerate} 
\end{lemma}
\begin{proof}
1)  We consider the units of the canonical adjoint pairs $(q:\mathcal D(R)\longrightarrow\mathcal{D}(\frac{R\Mod}{\mathcal T_Z}), \mathbf{R}j:\mathcal{D}(\frac{R\Mod}{\mathcal T_Z})\longrightarrow\mathcal D(R))$ and $(\mathbf{L}\mu^*=R_Z\otimes_R^\mathbf{L}?:\mathcal{D}(R)\longrightarrow\mathcal D(R_Z),\mu_*:\mathcal D(R_Z)\longrightarrow \mathcal{D}(R))$, which we denote by $\lambda :1_{\mathcal D(R)}\longrightarrow\mathbf{R}j\circ q$ and $\eta :1_{\mathcal D(R)}\longrightarrow\mu_*\circ\mathbf{L}\mu^*$, respectively. Note that, in order to avoid confusion,  we  divert from  Subsection 2.6, where we denoted by $\mu$ the unit $1_{R-\text{Mod}}\longrightarrow j\circ q$ of the 'abelian' adjunction $(q,j)$. Note also that, for each $M\in\mathcal D(R)$,  the triangle associated with the t-structure $(\mathcal{U}_Z,\mathcal{U}_Z^\perp =\mathcal{U}_Z^\perp[1])$ is  $\mathbf{R}\Gamma_Z(M)\longrightarrow M\stackrel{\lambda_M}{\longrightarrow} (\mathbf{R}j\circ q)(M)\stackrel{+}{\longrightarrow}$ (see \cite[Section 2]{AJSo2} and \cite[Section 1.6]{AJS}). By hypothesis, we have that $\mathbf{R}\Gamma_Z(M)\in\mathcal{U}$, which implies that  $M$ is in $\mathcal{U}$ if, and only if, so is $(Rj\circ q)(M)$. But $q\circ\mathbf{R}j$ is the identity on objects. We then get that $M\in\mathcal{U}$ if and only if $q(M)\cong q\circ\mathbf{R}j\circ q(M)\in q(\mathcal{U})$.

We next consider the unit map $\lambda_R:R\longrightarrow\mu_*(R_Z)=(\mu_*\circ\mathbf{L}\mu^*)(R)$ in $\mathcal D(R)$, which is just the ring homomorphism $\mu :R\longrightarrow R_Z$ viewed as morphism in $\mathcal{D}(R)$. 
Let us  complete it to a triangle $R\stackrel{\mu}{\longrightarrow}\mu_*(R_Z)\longrightarrow C\stackrel{+}{\longrightarrow}$ in $\mathcal D(R)$ (*). 
We then have $H^k(C)=0$, for $k\neq -1,0$, $H^{-1}(C)=\Ker (\mu )$ and $H^0(C)=\Coker (\mu )$, so that we also have a triangle $\text{Ker}(\mu )[1]\longrightarrow C\longrightarrow\text{Ker}(\mu)[0]\stackrel{+}{\longrightarrow}$ (**).   We claim that $\text{Supp}(H^k(C\otimes_R^\mathbf{L}M))\subseteq Z$, for each $M\in\mathcal{D}(R)$ and all $k\in\mathbb{Z}$. Since $\text{Ker}(\mu )$ and $\text{Coker}(\mu )$ are in $\mathcal{T}_Z$, using the triangle (**),  it is enough to check that $\text{Supp}(H^k(T\otimes_R^\mathbf{L}M))\subseteq Z$, for all  $T\in\mathcal{T}_Z$, $M\in\mathcal{D}(R)$ and $k\in\mathbb{Z}$. But, replacing $M$ by its homotopically projective resolution, we can assume that $M$ is homotopically projective (i.e. that $\text{Hom}_{\mathcal{K}(R)}(M,Y)=0$, for acyclic complex $Y$), in which case $T\otimes_R^\mathbf{L}M=T\otimes_RM$ is the usual tensor product. It is a standard fact that $\text{Supp}(T\otimes_RM^k)\subseteq \text{Supp}(T)\subseteq Z$,   for all $k\in\mathbb{Z}$ (see \cite[Exercise 3.19(iv)]{AM}), from which our claim follows immediately. 
 Moreover, we have an isomorphism $\mu_*(R_Z)\otimes_R^LM\cong (\mu_*\circ\mathbf{L}\mu^*)(M)$ in $\mathcal D(R)$. By applying $?\otimes_R^\mathbf{L}M$ to the triangle (*), we get a triangle in $\mathcal D(R)$ which is isomorphic to

$$M\stackrel{\eta_M}{\longrightarrow}(\mu_*\circ\mathbf{L}\mu^*)(M)\longrightarrow C\otimes_R^\mathbf{L}M\stackrel{+}{\longrightarrow}.$$
Arguing as in the previous paragraph, we deduce that $M\in\mathcal{U}$ if, and only if, $\mu_{*}(R_Z\otimes_R^\mathbf{L}M)=(\mu_{*}\circ\mathbf{L}\mu^{*})(M)\in\mathcal{U}$. This is turn is the case if, and only if, $(q\circ\mu_{*})(R_Z\otimes_R^\mathbf{L}M)$ is in $q(\mathcal{U})$. But the composition of exact functors $R_Z\Mod\stackrel{\mu_{*}}{\longrightarrow}R\Mod\stackrel{q}{\longrightarrow} \frac{R\Mod}{\mathcal T_Z}$ is naturally isomorphic to the localization functor $q':R_Z\Mod \longrightarrow  \frac{R\Mod}{\mathcal T_Z}$ and, since the three involved functors are exact, we also have a natural isomorphism $q\circ\mu_*\cong q'$ of their extensions to the respective derived categories. We then get that $M$ is in $\mathcal U$ if, and only if, $q'(R_Z\otimes_R^\mathbf{L}M)$ is in $q(\mathcal{U})$.  

\vspace*{0.3cm}

2) We put $\mathcal{U}=\mathcal{U}_\phi$, where $\phi$ is an sp-filtration of $\Spec(R)$. Note that the hypothesis implies that $Z\subseteq\bigcap_{i\in\mathbf{Z}}\phi (i)$. If $Y\in\mathcal{U}_\phi^\perp$, we then get a triangle $\mathbf{R}\Gamma_Z(Y)\stackrel{0}{\longrightarrow}Y\stackrel{\lambda_Y}{\longrightarrow} (\mathbf{R}j\circ q)(Y)\stackrel{+}{\longrightarrow}$ since $\mathcal{U}_Z\subseteq\mathcal{U}_\phi$. This immediately implies that $\lambda_Y$ is an isomorphism since $(\mathbf{R}j\circ q)(Y)\in\mathcal{U}_Z^\perp$. If now $X\in\mathcal{U}_\phi$ then, using adjunction,  we get a chain of isomorphisms: 

$$\Hom_{\D(\frac{R\Mod}{\T_Z})}(q(X),q(Y))\cong \Hom_{\D(R)}(X, (\mathbf{R}j \circ q)(Y))\cong \Hom_{\D(R)}(X,Y)=0.$$

 On the other hand, if $N \in \D(\frac{R\Mod}{\T_Z})$, then we get a triangle
$$\xymatrix{ \tau^{\leq}_{\phi}(\mathbf{R}j(N)) \ar[r] & \mathbf{R}j (N) \ar[r] & \tau^{>}_{\phi}(\mathbf{R}j(N)) \ar[r]^{\hspace{0.8 cm}+} & }$$
in $\D(R)$, from which we get a triangle in $\D(\frac{R\Mod}{\T_Z})$:
$$\xymatrix{q(\tau^{\leq}_{\phi}(\mathbf{R}j(N))) \ar[r] & q(\mathbf{R}j (N))\cong N \ar[r] & q(\tau^{>}_{\phi}(\mathbf{R}j(N))) \ar[r]^{\hspace{0.8 cm}+} & }$$ 
where the outer terms are in $q(\U)$ and $q(\U^{\perp})$ respectively.  This proves that $(q(\mathcal{U}),q(\mathcal{U}^\perp ) [1])$ is a t-structure in $\D(\frac{R\Mod}{\T_Z})$.

We now have a commutative diagram 
$$\xymatrix{\Ht \ar@{-->}[r]^{\tilde{q}\hspace{0.15 cm}} \ar@{^(->}[d]& \mathcal{H}_{q(\phi)} \ar@{^(->}[d]\\ \D(R) \ar[r]^{q\hspace{0.3 cm}} & \D(\frac{R\Mod}{\T_Z})}$$
We shall prove that the upper horizontal arrow is an equivalence of categories. We will start by showing that $\tilde{q}$ is dense. Let $M\in \mathcal{H}_{q(\phi)}$, we then have $M\cong q(U)$, for some $U\in \U$ and, by the first paragraph of this proof, we can and shall assume that $\lambda_U:U \flecha (\mathbf{R}j \circ q) (U)$ is an isomorphism. We claim that  $U\in \Ht$, for which we just need to see that $U[-1]\in \U^{\perp}$. Indeed, for each $U^{'}\in U$, we have the following isomorphisms $\Hom_{\D(R)}(U^{'},U[-1])\cong \Hom_{\D(R)}(U^{'},(\mathbf{R}j \circ q)(U)[-1]) \cong \Hom_{\D(\frac{R\Mod}{\T_Z})}(q(U^{'}),q(U)[-1])=0$. This is zero since $q(U^{'}) \in q(\U)$ and $q(U)\in \mathcal{H}_{q(\phi)}$. \\

Finally, if $M,N$ are in $\Ht$, then we get an isomorphism $\lambda_{N[-1]}:N[-1] \iso (\mathbf{R} j \circ q)(N[-1])$ (since $N[-1]\in \U^{\perp}$). Moreover, we have the following chain of isomorphisms:
$$\xymatrix{\Hom_{\D(R)}(M,N) \ar[r]^{\sim \hspace{0.75cm}} & \Hom_{\D(R)}(M[-1],N[-1]) \ar[r]^{(\lambda_{N[-1]})_{*} \hspace{0.75 cm}} & \Hom_{\D(R)}(M[-1],(\mathbf{R}j \circ q)(N[-1])) \ar[d]^{\wr} \\ & & \Hom_{\D(\frac{R\Mod}{\T_Z})}(q(M[-1]), q(N[-1])) \ar[d]^{\wr} \\ & & \Hom_{\D(\frac{R\Mod}{\T_Z})}(q(M), q(N))} $$
It is not difficult to see that the resulting isomorphism coincides with $\xymatrix{\Hom_{\D(R)}(M,N) \ar[r]^{\tilde{q}(?) \hspace{0.8 cm}} & \Hom_{\D(\frac{R\Mod}{\T_Z})}(q(M),q(N))}$. Therefore $\tilde{q}$ is full and faithful.
\end{proof}

\begin{defi}
An sp-subset $Z\subseteq\Spec (R)$ will be called \emph{perfect} when the ring of fractions $R_Z$ is flat as an $R$-module  and $Z$ consists of the prime ideals $\mathbf{p}$ such that $\mu (\mathbf{p})R_Z=R_Z$.
A (commutative) ring $A$ will be called a \emph{perfect localization} of $R$ when there is a perfect sp-subset $Z$ of $\Spec (R)$ such that $A\cong R_Z$.  
\end{defi}

\begin{rem} \label{rem.perfect localization}
An sp-subset $Z\subseteq\Spec (R)$ is perfect if, and only if, the canonical functor $j':\frac{R\Mod}{\mathcal T_Z}\longrightarrow R_Z\Mod$ is an equivalence of categories, which in turn is equivalent to saying that $\mu_*^{-1}(\mathcal T_Z)=0$ (see \cite[Chapter XI]{S}).   
\end{rem}


\begin{prop} \label{prop.perfect localizations}
Let $Z$ be an sp-subset of $\Spec (R)$. The following assertions are equivalent:

\begin{enumerate}
\item $R\Mod/\mathcal T_Z$ is a module category;
\item $Z$ is  perfect.
\end{enumerate}
\end{prop}
\begin{proof}
$2)\Longrightarrow 1)$ is clear.

$1)\Longrightarrow 2)$ As usual, whenever necessary, we identify $R\Mod/\mathcal T_Z$ with the Giraud subcategory $\mathcal{G}_Z$. Then the lattice of subobjects of $R_Z$ in this category 'is' a sublattice of the lattice of ideals of $R$ (see \cite[Proposition IX.4.3 and Corollary IX.4.4]{S}). In particular, $R_Z$ is a Noetherian generator of $\mathcal{G}_Z$. 

Let now $P$ be a progenerator of $\mathcal{G}_Z$. Then we have an epimorphism $\epsilon :P^n\longrightarrow R_Z$ in $\mathcal{G}_Z$, for some integer $n>0$. Without loss of generality, we can assume that $n=1$. Viewed as a morphism in $R_Z\Mod$, we have that $\Coker (\epsilon )$ is a torsion module (i.e. it is in $\mu_*^{-1}(\mathcal T_Z)$). On the other hand, we have an epimorphism $p:R_Z^m\longrightarrow P$ in $\mathcal{G}_Z$, for some integer $m>0$. Then $p$ is a retraction in this category, which implies that it is also a retraction in $R_Z\Mod$. It follows that $P$ is a  finitely generated projective $R_Z$-module. 

Note that,  by Lemma \ref{lem.Ext-orthogonal}, we have an equality of full subcategories of $R\Mod$:

$$\mathcal{G}_Z=(\bigcap_{\mathbf{p}\in Z}\text{Ker}(\Hom_R(R/\mathbf{p},?)))\cap (\bigcap_{\mathbf{p}\in Z}\text{Ker}(\Ext_R^1(R/\mathbf{p},?)). $$
It follows that $\mathcal{G}_Z$ is closed under taking coproducts in $R\Mod$ and, hence, also in $R_Z\Mod$. That is, the inclusion $\mathcal{G}_Z\hookrightarrow R_Z\Mod$ preserves coproducts. Let then consider the canonical morphism of $R_Z$-modules $g:P^{(\Hom_{R_Z}(P,R_Z))}\longrightarrow R_Z$, whose image $\mathbf{a}:=\Im (g)$ is the trace of $P$ in $R_Z$ and, hence, it is an idempotent ideal of $R_Z$. Since $\epsilon$ is an epimorphism in $\mathcal{G}_Z$ and $P^{(\Hom_{R_Z}(P,R_Z))}$ is a projective object of this category, we get that $g$ factors through $\epsilon$, so that we have a morphism $f:P^{(\Hom_{R_Z}(P,R_Z))}\flecha P$ such that $\epsilon\circ f=g$. It then follows that 

$$\Im (\epsilon )\subseteq tr_P(R_Z)=\mathbf{a}=\Im (g)=\Im (\epsilon\circ f)\subseteq\Im (\epsilon ),$$
which implies that all these inclusions are equalities. We then get that $\mathbf{a}=\Im (\epsilon )$, so that   $\mathbf{a}$ is a finitely generated ideal of $R_Z$.   By \cite[Lemma VI.8.6]{S}, we know that   $\mathbf{a}=R_Ze$, for some idempotent $e\in R_Z$. But then $R_Z(1-e)\cong\Coker (\epsilon )$ is a torsion ideal of $R_Z$, whence necessarily zero. Therefore we get that $e=1$, so that $\epsilon$ is also an epimorphism in $R_Z\Mod$. We then have that $\text{add}(R_Z)=\text{add}(P)$, so that $R_Z$ is also a progenerator of $\mathcal{G}_Z$. Bearing in mind that we have a ring isomorphism $R_Z\cong\End_{\mathcal{G}_Z}(R_Z)$ and that $R_Z$ is commutative, we get an equivalence of categories:

$$\Hom_{R_Z}(R_Z,?)=\Hom_{\mathcal{G}_Z}(R_Z,?):\mathcal{G}_Z \iso \End_{\mathcal{G}_Z}(R_Z)\Mod\stackrel{\cong}{\longleftrightarrow}R_Z\Mod.$$ 
Clearly, this functor is naturally isomorphic to the inclusion functor $\mathcal{G}_Z\longrightarrow R_Z\Mod$ and, by Remark \ref{rem.perfect localization}, we conclude that $Z$ is perfect. 
\end{proof}

We are ready to prove the second main result of the paper. In its statement and proof,   $\Supp (M)$ will mean always the support of $M$ as an $R$-module, even if $M$ is an $R_Z$-module. 

\begin{teor} \label{teor.modular main theorem1}
Let $R$ be a commutative Noetherian ring, let $(\mathcal{U},\mathcal{U}^\perp [1])$ be a compactly generated t-structure in $\mathcal D(R)$ such that $\mathcal{U}\neq\mathcal{U}[-1]$ and let $\mathcal{H}$ be its heart. The following assertions are equivalent:
\begin{enumerate}
\item[1)] $\mathcal{H}$ is a module category;
\item[2)] There are a possibly empty stable under specialization subset $Z\subsetneq\Spec(R)$, a family  $\{e_1,\dots,e_t\}$ of nonzero orthogonal idempotents of the ring of fractions $R_Z$  and integers $m_1<m_2< \dots<m_t$ such that:

\begin{enumerate}
\item If $N$ is a $R_Z$-module such that $\Supp (N)\subseteq Z$, then $e_kN=0$ for $k=1,\dots,t$;

\item A complex $U\in\mathcal D(R)$ is in $\mathcal{U}$ if, and only if,   $\Supp (e_0H^j(R_Z\otimes^\mathbf{L}_RU))\subseteq Z$  for each $j\in\mathbb{Z}$, where $e_0=1-\sum_{1\leq k\leq t}e_k$, and 
$e_k(R_Z\otimes_R^\mathbf{L}U)\in\mathcal{D}^{\leq m_k}(R_Ze_k)$ for all $k=1,\dots,t$.
\end{enumerate}
	\end{enumerate}
In that case,  we have an equivalence of categories $\mathcal{H}\stackrel{\cong}{\longrightarrow}A\Mod$, where $A=\oplus_{1\leq k\leq t}R_Ze_k=R_Z(1-e_0)$, and $A$ is a perfect localization of $R$. 
\end{teor}
\begin{proof}
$2)\Longrightarrow 1)$ Our hypotheses  imply the existence of a decomposition 
$$\frac{R\Mod}{\mathcal T_Z}\cong\frac{R_Z\Mod}{\mathcal\mu_{*}^{-1}(T_Z)}\cong \frac{R_Ze_0\Mod}{\mu_*^{-1}(T_Z)}\times R_Ze_1\Mod\times \dots \times R_Ze_t\Mod, $$ and this decomposition passes to the corresponding derived categories

$$\mathcal D(\frac{R\Mod}{\mathcal T_Z})\cong\mathcal D(\frac{R_Z\Mod}{\mathcal\mu_*^{-1}(T_Z)})\cong\mathcal D( \frac{R_Ze_0\Mod}{\mu_*^{-1}(T_Z)})\times\mathcal D(R_Ze_1\Mod)\times \dots \times\mathcal D(R_Ze_t\Mod).$$

If $q:\mathcal{D}(R)\longrightarrow\mathcal{D}(\frac{R\Mod}{\mathcal T_Z})$ is the canonical functor, then we see that $q(\mathcal{U})$ is the direct sum of aisles $0\oplus \underset{1\leq k\leq t}{\oplus}\mathcal{D}^{\leq m_k}(R_Ze_k\Mod)=0\oplus \underset{1\leq k\leq t}{\oplus}\mathcal{D}^{\leq m_k}(R_Ze_k)$ with respect to the last decomposition. It follows from Lemmas \ref{lem. product of aisles} and \ref{lem. q(U) is a t-struc.} that $\mathcal{H}$ is equivalent to $R_Ze_1\Mod\times \dots \times R_Ze_t\Mod\cong A\Mod$, where $A=R_Z(1-e_0)$.

Although not needed for the proof of this implication, we will check that $A=R_Z(1-e_0)$ is a perfect localization of $R$. For that, we consider the canonical  compositions of functors 
$$R\Mod\stackrel{q}{\longrightarrow}\frac{R\Mod}{\mathcal T_Z}\cong\frac{R_Z\Mod}{\mathcal \mu_*^{-1}(T_Z)}\stackrel{proj.}{\longrightarrow}R_Ze_1\Mod\times \dots \times R_Ze_t\Mod\cong A\Mod$$ and 

$$R\Mod\stackrel{q}{\longrightarrow}\frac{R\Mod}{\mathcal T_Z}\cong\frac{R_Z\Mod}{\mathcal \mu_*^{-1}(T_Z)}\stackrel{proj.}{\longrightarrow}\frac{R_Ze_0\Mod}{\mu_{*}^{-1}(\mathcal T_Z)}.$$ Their kernels, which we denote by  $\mathcal T_1$ and $\mathcal T_2$, are hereditary torsion classes in $R\Mod$ which contain $\mathcal{T}_Z$ and satisfy that $\frac{\mathcal T_1}{\mathcal T_Z}\cong\frac{R_Ze_0\Mod}{\mu_{*}^{-1}(\mathcal T_Z)}$ and $\frac{\mathcal T_2}{\mathcal T_Z}\cong R_Ze_1\Mod\times \dots\times R_Ze_t\Mod\cong A\Mod$, so that we also have  
$\frac{R\Mod}{\mathcal T_Z}\cong\frac{\mathcal T_1}{\mathcal T_Z}\times\frac{\mathcal T_2}{\mathcal T_Z}$. Lemma \ref{lem. product of quotient category} and its proof show that then $\frac{R\Mod}{\mathcal T_1}\cong A\Mod$. By Proposition \ref{prop.perfect localizations} , we conclude that if $R_1$ denotes the ring of fractions of $R$ with respect to the torsion pair $(\mathcal T_1,\mathcal T_1^\perp)$, then $R_1$ is a (commutative) ring Morita equivalent to $A$. But Morita equivalent commutative rings are isomorphic. 

$1)\Longrightarrow 2)$  By Lemma \ref{lem. product of aisles}, we can assume, without loss of generality, that $R$ is connected. Let $\phi$ be the filtration by supports of $\Spec(R)$ associated to  $(\mathcal{U},\mathcal{U}^\perp [1])$ and put  $Z:=\underset{i\in I}{\bigcap} \phi(i)$. Let $\mathbf{p}\in\text{Spec}(R)\setminus Z$ and  $i\in\mathbb{Z}$ be given, and suppose that $\mathbf{p}\in\phi (i)$. Identifying $\text{Spec}(R_\mathbf{p})=\{\mathbf{q}\in\text{Spec}(R)\text{: }\mathbf{q}\subseteq\mathbf{p}\}$, we claim that this set is contained in $\phi (i)$, so that $\phi (i)\setminus Z$ is stable under generalization (and specialization) within $\Spec (R)\setminus Z$. Without loss of generality, we assume that $i$ is maximal with the property that $\mathbf{p}\in\phi (i)$. By considering the associated sp-filtration $\phi_\mathbf{p}$ on $\text{Spec}(R_\mathbf{p})$, we then have that $\phi_\mathbf{p}(i)\neq\emptyset =\phi_\mathbf{p}(i+1)$. 
 From Corollary \ref{cor.modular heart via localization} and Proposition \ref{prop.modular heart connected case} we get that $\phi_\mathbf{p}(i)=\Spec (R_\mathbf{p})$, which settles our claim. 
 Let now $\text{MinSpec}(R)$ be the set of minimal prime ideals of $R$ and put  $\mathcal{M}_i=\text{MinSpec (R)}\cap\phi (i)$, for each $i\in\mathbb{Z}$. We  then have  $\phi (i)=\bar{\mathcal{M}}_i\cup Z$, where the upper bar denotes the Zariski closure in $\Spec (R)$. Bearing in mind that $\text{MinSpec (R)}$ is a finite set, we get integers $m_1<m_2<\cdots<m_t$ satisfying the following properties:

\begin{enumerate}
\item[i)] $\phi (i)\supsetneq\phi (i+1)$ if, and only if, $i\in\{m_1,\dots,m_t\}$. Moreover $\phi (i)=Z$, for all $i>m_t$.
\item[ii)] If we put $\tilde{V}_0=\Spec (R)\setminus\phi (m_1)$, $\tilde{V}_t=\phi (m_t)\setminus Z$ and $\tilde{V}_k=\phi (m_k)\setminus\phi (m_{k+1})=\phi (m_k)\setminus\phi (m_k+1)$, for $k=1,\dots,t-1$, then each $\tilde{V}_k$ is stable under specialization and generalization in $\Spec (R)\setminus Z$ and we have $\Spec (R)\setminus Z=\underset{0\leq k\leq t}{\bigcupdot}\tilde{V}_k$.
\end{enumerate}

By Lemma \ref{lem. q(U) is a t-struc.}, if $q:\mathcal{D}(R)\longrightarrow\mathcal{D}(R\Mod/\mathcal{T}_Z)$ is the canonical functor, then $(q(\mathcal{U}),q(\mathcal{U}^\perp )[1])$ is a t-structure in $\mathcal{D}(R\Mod/\mathcal{T}_Z)$ whose heart $\mathcal{H}_{q(\phi)}$ is equivalent to $\mathcal{H}=\Ht$. On the other hand, an iterative use of Lemma \ref{lem. product of quotient category} says that if $V_k=\tilde{V}_k\cup Z$,  for each $k=0,1,\dots,t$, then we have an equivalence of categories

$$\frac{R\Mod}{\mathcal{T}_Z}\stackrel{\cong}{\longleftrightarrow} \frac{\mathcal{T}_{V_0}}{\mathcal{T}_Z}\times \frac{\mathcal{T}_{V_1}}{\mathcal{T}_Z}\times \cdots\times \frac{\mathcal{T}_{V_t}}{\mathcal{T}_Z}.$$

If now $R_Z$ is the ring of fractions of $R$ with respect to $Z$, then, when viewed as an object of $R\Mod/\mathcal{T}_Z$, it decomposes in a direct sum $R_Z=Y_0\oplus Y_1\oplus \cdots \oplus Y_t$, where $Y_k\in\tilde{\T_{V_k}}:= \frac{\mathcal{T}_{V_k}}{\mathcal{T}_Z}$ for each $k=0,1,\dots,t$. This decomposition corresponds to a decomposition of the identity of $\text{End}_{R\Mod/\mathcal{T}_Z}(R_Z)$ as a sum of orthogonal idempotent endomorphisms. But we have a ring isomorphism $R_Z \iso \text{End}_{R\Mod/\mathcal{T}_Z}(R_Z)$ since the section functor $R\Mod/\mathcal{T}_Z\longrightarrow R_Z\Mod$ is fully faithful (see \cite[Chapter IX, Section 1]{S}). We then get idempotents $e_0,e_1,\dots,e_t\in R_Z$, which are central since $R_Z$ is commutative, such that $1=e_0+e_1+\dots+e_t$. Note that all the $e_k$ are nonzero except perhaps $e_0$, which is zero exactly when $\phi (m_1)=\Spec(R)$. 

Using the equivalence of categories $R_Z\Mod/\mu_*^{-1}(\mathcal{T}_Z)\stackrel{\cong}{\longleftrightarrow}R\Mod/\mathcal{T}_Z$ given by the restriction of scalars $\mu_*:R_Z\Mod\longrightarrow R\Mod$, we can now rewrite the above decomposition of Grothendieck categories as

$$\frac{R_Z\Mod}{\mu_{*}^{-1}(\mathcal{T}_Z)}\stackrel{\cong}{\longleftrightarrow}\frac{R_Ze_0\Mod}{\mu_{*}^{-1}(\mathcal T_Z)\cap R_Ze_0\Mod}\times\frac{R_Ze_1\Mod}{\mu_{*}^{-1}(\mathcal T_Z)\cap R_Ze_1\Mod}\times \cdots \times\frac{R_Ze_t\Mod}{\mu_{*}^{-1}(\mathcal T_Z)\cap R_Ze_t\Mod}.$$
That is, each category $\frac{\mathcal{T}_{V_k}}{\mathcal{T}_Z}$ is identified with the full subcategory of $Z$-closed $R$-modules that, when viewed as $R_Z$-modules, belong to $R_Ze_k\Mod$ (equivalently, are annihilated  $1-e_k$). 

The last decomposition of Grothendieck categories passes to the corresponding derived categories. Moreover, by Lemma \ref{lem. q(U) is a t-struc.}, we know that $X\in\mathcal D(R)$ is in $\mathcal{U}$ if, and only if, $q(X)\in q(\mathcal{U})$. We claim that  $q(\mathcal{U})$ is the direct sum of aisles $0\oplus \underset{1\leq k\leq t}{\oplus}\mathcal{D}^{\leq m_k}(\frac{R_Ze_k\Mod}{\mu_{*}^{-1}(\mathcal T_Z)\cap R_Ze_k\Mod})$, with respect to last decomposition. In order to prove this, in the sequel, we will keep $H^i(?)$ to denote the $i$-th homology $R$-module and will denote by $H_{\mathcal{G}_Z}^i(?)$ the corresponding homology object in the category $\mathcal{G}_Z\cong\frac{R\Mod}{\mathcal{T}_Z}$. 
Let $Y\in\mathcal D(\mathcal{G}_Z)$, which we view as a homotopically injective object of the homotopy category  $\mathcal{K}(\mathcal G)$ of $\mathcal G $ (i.e. we assume that $\Hom_{\mathcal{K}(\mathcal{G})}(?,Y)$ vanishes on acyclic complexes), so that $\mathbf{R}j(Y)=Y$ since $j:\mathcal{G}_Z\longrightarrow R\Mod$ is the inclusion functor. That is,  $\mathbf{R}j(Y)$ is the same complex $Y$, viewed as a complex of $R$-modules.  By Lemma \ref{lem. q(U) is a t-struc.}, we know that $Y\in q(\mathcal{U})$ if, and only if,  $\mathbf{R}j(Y)=Y\in\mathcal{U}$ when viewed as a complex of $R$-modules. Then $Y\in q(\mathcal{U})$ if, and only if, $\Supp (H^i(Y))\subseteq\phi (i)=\phi (m_s)=Z\cup\tilde{V}_s\cup\cdots \cup\tilde{V}_t\cup\tilde{V}_{t+1}$, whenever $m_{s-1}<i\leq m_s$, with the convention that $m_0=-\infty$ and $m_{t+1}=+\infty$ and that $\tilde{V}_{t+1}=\emptyset$. When applying the functor $q:R\Mod\longrightarrow\mathcal{G}_Z$, this is equivalent to saying that $H_{\mathcal{G}_Z}^i(Y)$ is in $\frac{\mathcal{T}_{V_s}}{\mathcal{T}_Z}\oplus\frac{\mathcal{T}_{V_{s+1}}}{\mathcal{T}_Z}\oplus \cdots\oplus\frac{\mathcal{T}_{V_t}}{\mathcal{T}_Z}$ whenever $m_{s-1}<i\leq m_s$. Using now the equivalences of categories $\frac{\mathcal{T}_{V_k}}{\mathcal{T}_Z}\cong\frac{R_Ze_k\Mod}{\mu_{*}^{-1}(\mathcal T_Z)\cap R_Ze_k\Mod}$, with $k=0,1,\dots,t$, we conclude that $Y$ is in $q(\mathcal{U})$ if, and only if, we have $e_0Y=0$ and $H_{\mathcal{G}_Z}^i(e_kU)=0$, for all $k=1,\cdots,t$ and all $i>m_k$. This is exactly saying that $e_0Y=0$ and $e_kY\in\mathcal{D}^{\leq m_k}(\frac{R_Ze_k\Mod}{\mu_{*}^{-1}(\mathcal T_Z)\cap R_Ze_k\Mod})$, for each $k=1,\dots,t$. Therefore we have the desired equality of aisles $q(\mathcal{U})=0\oplus \underset{1\leq k\leq t}{\oplus}\mathcal{D}^{\leq m_k}(\frac{R_Ze_k\Mod}{\mu_{*}^{-1}(\mathcal T_Z)\cap R_Ze_k\Mod})$.

Bearing in mind that the heart of $(\mathcal{D}^{\leq m_k}(\frac{R_Ze_k\Mod}{\mu_{*}^{-1}(\mathcal T_Z)\cap R_Ze_k\Mod}),\mathcal{D}^{\geq m_k}(\frac{R_Ze_k\Mod}{\mu_{*}^{-1}(\mathcal T_Z)\cap R_Ze_k\Mod}))$ is equivalent to $\frac{R_Ze_k\Mod}{\mu_{*}^{-1}(\mathcal T_Z)\cap R_Ze_k\Mod}$, the hypothesis together with Lemmas \ref{lem. q(U) is a t-struc.}  and \ref{lem. product of aisles}  give that $\frac{R_Ze_k\Mod}{\mu_{*}^{-1}(\mathcal T_Z)\cap R_Ze_k\Mod}$ is a module category, for each $k=1,\dots,t$. Then, by Proposition \ref{prop.perfect localizations} and the fact that $R_Ze_k$ is $Z$-closed, we get that $\mu_{*}^{-1}(\mathcal T_Z)\cap R_Ze_k\Mod=0$, for $k=1,\dots,t$. This gives condition 2.a in the statement of the theorem. On the other hand, using Lemma \ref{lem. q(U) is a t-struc.} again, we have  the following chain of implications for $U\in\mathcal D(R)$:

$$U\in\mathcal{U}\Longleftrightarrow q'(R_Z\otimes_R^\mathbf{L}U)\in q(\mathcal{U})\cong 0\times\mathcal D^{\leq m_1}(R_Ze_1)\times ...\times D^{\leq m_t}(R_Ze_t),$$ 
when considering the decomposition $\mathcal{D}(\frac{R\Mod}{\mathcal T_Z})\cong\mathcal{D}(\frac{R_Ze_0\Mod}{\mu_{*}^{-1}(\mathcal T_Z)})\times\mathcal D(R_Ze_1)\times \dots \times\mathcal D(R_Ze_t)$ seen above. Then $U$ is in $\mathcal{U}$ if and only if $e_k(R_Z\otimes_R^{\mathbf{L}}U)\in\mathcal D^{\leq m_k}(R_Ze_k)$, for $k=1,\dots,t$, and $q'(e_0(R_Z\otimes_R^{\mathbf{L}}U))=0$. But this last equality holds if, and only if, all the homology $R_Z$-modules of $e_0(R_Z\otimes_R^\mathbf{L}U)$ are in $\mu_{*}^{-1}(\mathcal T_Z)$. That is, if and only if, their support is in $Z$ when viewed as $R$-modules. 
\end{proof}

The following example shows that, even when $R$ is connected, the heart of a compactly generated t-structure can be a module category which strictly decomposes as a product of smaller (module) categories.

\begin{exem}
Let $R$ be reduced (i.e. with no nonzero nilpotent elements), let $\{\mathbf{p}_1,\dots,\mathbf{p}_r\}$ be a subset of $\text{MinSpec}(R)$ and consider the sp-filtration $\phi$ of $\Spec(R)$ given as follows:

\begin{enumerate}
 \item $\phi (i)=\Spec (R)\setminus\text{MinSpec}(R)$, for all $i\geq 0$;
\item $\phi (i)=(\Spec (R)\setminus\text{MinSpec}(R))\cup\{\mathbf{p}_1,\dots,\mathbf{p}_{-i}\}$, for $-r< i< 0$
\item $\phi (i)=(\Spec (R)\setminus\text{MinSpec}(R))\cup\{\mathbf{p}_1,\dots,\mathbf{p}_{r}\}$, for all $i\leq -r$.
\end{enumerate}
If $(\mathcal{U}_\phi ,\mathcal{U}_\phi^\perp [1])$ is the associated compactly generated t-structure of $\mathcal{D}(R)$, then the heart $\Ht$ is equivalent to $k(\mathbf{p}_1)\Mod\times \cdots \times k(\mathbf{p}_r)\Mod\cong (k(\mathbf{p}_1)\times \cdots \times k(\mathbf{p}_r))\Mod$, where $k(\mathbf{p})$ denotes the residue field at $\mathbf{p}$, for each $\mathbf{p}\in\Spec (R)$. 
\end{exem}
\begin{proof}
If $S$ denotes the set of nonzero divisors of $R$ and $Z=\Spec (R)\setminus\text{MinSpec}(R)$, then $R_Z\cong S^{-1}R$ and we have a ring isomorphism $S^{-1}R\iso \prod_{\mathbf{p}\in\text{MinSpec}(R)}k(\mathbf{p})$ (see \cite[Proposition III.4.23]{Ku}). Moreover $S^{-1}R$ is a perfect localization of $R$ (see \cite[Proposition XI.6.4]{S}). It follows that $R\Mod/\mathcal{T}_Z$ is equivalent to $S^{-1}R\Mod$ and the canonical functor $q:R\Mod\longrightarrow R\Mod/\mathcal{T}_Z$ gets identified with the localization functor $S^{-1}(?)\cong S^{-1}R\otimes_R?:R\Mod\longrightarrow S^{-1}R\Mod$. In this case the ring homomorphism  $\mu :R\longrightarrow S^{-1}R$  is the canonical one, and we have $\mu_*^{-1}(\mathcal T_Z)=0$, so that condition 2.a of Theorem \ref{teor.modular main theorem1} is automatic.

Denote now by $e_j$ the idempotent of $S^{-1}R$ corresponding to the summand $k(\mathbf{p}_j)$ of $S^{-1}R$, for each $j=1,\dots,r$. Given a complex $U\in\mathcal D(R)$, we have that $U\in\mathcal{U}_\phi$ if, and only if, $S^{-1}U=S^{-1}R\otimes_RU=S^{-1}R\otimes_R^\mathbf{L}U\in\mathcal{U}$ (see Lemma \ref{lem. q(U) is a t-struc.}). This happens if, and only if,  $e_0S^{-1}U=0$, where $e_0=1-\underset{1\leq j\leq r}{\sum}e_j$, and $e_jS^{-1}R\leq D^{\leq -j}(S^{-1}Re_j)=D^{\leq -j}(k(\mathbf{p}_j))$, for each $j=1,\dots,r$. Then condition 2.b of the theorem also holds. 
\end{proof}

Some direct  consequences can be derived now from last theorem and its proof. 

\begin{cor} \label{cor.modular case for left non-degenerate}
Let $R$ be connected and let $(\mathcal{U},\mathcal{U}^\perp [1])$ be a compactly generated left nondegenerate t-structure in $\mathcal{D}(R)$ such that $\mathcal{U}\neq 0$. The heart of this t-structure is a module category if, and only if, $(\mathcal{U},\mathcal{U}^\perp [1])=(\mathcal{D}^{\leq m}(R),\mathcal{D}^{\geq m}(R))$, for some integer $m$. 
\end{cor}
\begin{proof}
If $\phi$ is the associated sp-filtration, then the left nondegeneracy of $(\mathcal{U},\mathcal{U}^\perp [1])$ translates into the fact that $\bigcap_{i\in\mathbb{Z}}\phi(i)=\emptyset$. We now apply Theorem \ref{teor.modular main theorem1}. The proof of its implication $1)\Longrightarrow 2)$ shows that the sp-subset $Z$ given by its assertion 2 is empty in this case. Since $R_Z=R$ is connected, we have a unique nonzero idempotent $e_1=1$ and a unique integer $m=m_1$ in that assertion 2.  The result is then a straightforward consequence of the mentioned theorem.  
\end{proof}

\begin{cor} \label{cor.modular for prime nilradical}
Let $R$ be connected and suppose that its nilradical is a prime ideal. Let  $(\mathcal{U},\mathcal{U}^\perp [1])$ be a compactly generated t-structure in $\mathcal{D}(R)$ such that $\mathcal{U}\neq\mathcal{U}[-1]$. The heart $\mathcal{H}$ of this t-structure is a module category if, and only if there are a possibly empty perfect sp-subset $Z\subseteq\Spec (R)$ and an integer $m$ such that the following property holds:

($\dagger$) A complex $U\in\mathcal{D}(R)$ is in $\mathcal{U}$ if, and only if, $R_Z\otimes_R U$ is in $\mathcal D^{\leq m}(R_Z)$ (equivalently, if, and only if, $\Supp (H^j(U))\subseteq Z$ for all $j>m$). 

 In this case $\mathcal{H}$ is equivalent to $R_Z\Mod$.
\end{cor}
\begin{proof}
The 'if part' is a direct consequence of Theorem \ref{teor.modular main theorem1}. For the 'only if' part, note that the proof of  implication $1)\Longrightarrow 2)$ in the mentioned theorem gives an sp-subset $Z\subsetneq\Spec (R)$ and  decomposition $\Spec (R)\setminus Z=\underset{0\leq k\leq t}{\bigcupdot}\tilde{V}_k$, where the $\tilde{V}_k$ are stable under specialization and generalization in $\Spec (R)\setminus Z$, and all $\tilde{V}_k$ are nonempty except perhaps $\tilde{V}_0$, which is empty exactly when $\phi (i)=\Spec (R)$ for some
 $i\in\mathbb{Z}$. The existence of a unique minimal prime ideal of $R$ implies that $t=1$ and, with the terminology of Theorem \ref{teor.modular main theorem1}, that there is a unique integer $m=m_1$ in its assertion 2. Moreover, we have $\tilde{V}_0=\emptyset$.  Then the associated filtration by supports satisfy that $\phi (i)=\Spec (R)$, for $i\leq m$, and $\phi (i)=Z$, for $i>m$. By Proposition \ref{prop.heart equivalent to quotient}, we get that $R\Mod/\mathcal T_Z$ is a module category and, by Proposition \ref{prop.perfect localizations},  we get that $Z$ is a perfect sp-subset. But then  we have that  $R_Z\otimes_R^\mathbf{L}U=R_Z\otimes_RU$ since $R_Z$ is a flat $R$-module. It follows that a complex $U\in\mathcal D(R)$ is in $\mathcal{U}=\mathcal{U}_\phi$ (i.e. satisfies that $\Supp (H^i(U))\subseteq Z$, for all $i>m$) if, and only if, $R_Z\otimes_RU\in\mathcal{D}^{\leq m}(R_Z)$. 
\end{proof}

\begin{proof}[Proof of Theorem B]:  Assertion 2.a of this theorem is a geometric translation of Corollary \ref{cor.modular case for left non-degenerate}. On the other hand, given an affine noetherian scheme $\mathbb{X}=\Spec(R)$, the assignment $Z\rightsquigarrow\mathbb{X}\setminus Z$ defines a bijection between perfect sp-subsets of $\mathbb{X}$ and affine subschemes $\mathbb{Y}\subseteq\mathbb{X}$
(see \cite[Teorema 1.1.20]{J}, whose proof is deduced from \cite[Proposition 2.5]{L} and \cite[Proposition 6.13]{V}). Now assertions 2.b and 1 in Theorem B are geometric translations of Corollary \ref{cor.modular for prime nilradical} and the final statement of Theorem \ref{teor.modular main theorem1} 
\end{proof}


\begin{thebibliography}{80}

\bibitem[AJS]{AJS} L. Alonso, A. Jerem\'ias, M. Saor\'in, \emph{Compactly generated t-structures on the derived category of a Noetherian ring,} Journal of Algebra, \textbf{324} (2010), 313-346.

\bibitem[AJSo]{AJSo} L. Alonso, A. Jerem\'ias, M.J. Souto, \emph{Constructions of t-structures and equivalences of derived categories.} Trans. Amer. Math. Soc. \textbf{355}(6)  (2003), 2523-2543.

\bibitem[AJSo2]{AJSo2} L. Alonso, A. Jerem\'ias, M.J. Souto, \emph{Bousfield localization on formal shemes}, J. Algebra, \textbf{278}  (2004), no. 2, 585-610.

\bibitem[AM]{AM} M. Atiyah, I. MacDonald, \emph{Introduction to Commutative Algebra}, Addison-Wesley Series in Mathematics, (1969).


\bibitem[B]{B} A.A. Beilinson, \emph{Coherent sheaves on $\mathbb{P}^n$ and problems of linear algebra}, Funct. Anal. \textbf{12}(3) (1978), 214-216.

\bibitem[BBD]{BBD} A. Beilinson, J. Bernstein, P. Deligne, ``Faisceaux Pervers''. \emph{Analysis and topology on  singulas spaces,} I, Luminy 1981, Ast\`{e}risque. \textbf{100}.
  Soc. Math. France, Paris. (1982), 5-171.
  
  
\bibitem[Br]{Br} T. Bridgeland, \emph{Stability conditions on triangulated categories}. Annals of Math. \textbf{166} (2007), 317-345.

\bibitem[CN]{CN} C. Casacuberta, A. Neeman, \emph{Brown representability does not come for free. Math. Res. Lett.} \textbf{16}(1) (2009), 1-5.


\bibitem[CG]{CG} R. Colpi, E. Gregorio, \emph{The heart of a cotilting torsion pair is a Grothendieck category}, preprint.


\bibitem[CGM]{CGM} R. Colpi, E. Gregorio, F. Mantese, \emph{On the Heart of a faithful torsion theory}, Journal of Algebra,
 \textbf{307} (2007), 841-863.


\bibitem[CMT]{CMT}R. Colpi, F. Mantese, A. Tonolo, \emph{When the heart of a faithful torsion pair is a module category}, Journal of Pure and Applied Algebra, \textbf{215} (2011) 2923-2936.


\bibitem[G]{G} P. Gabriel, \emph{Des cat\'egories abeliennes}, Bull. Soc. Math. France \textbf{90} (1962), 323-448.

\bibitem[GL]{GL} W. Geigle, H. Lenzing, \emph{Perpendicular categories with applications to representations and sheaves}. J. Algebra \textbf{144} (1991), 273-343.


\bibitem[GT]{GT} R. G\"obel, J. Trlifaj, \emph{Approximations and endomorphism algebras of
modules}. Expositions in Maths \textbf{41}. De Gruyter (2006)


\bibitem[GKR]{GKR} A.L. Gorodentsev, S.A. Kuleshov, A.N. Rudakov, \emph{t-stabilities and t-structures on triangulated categories}. Izvest. RAN Ser. Math. \textbf{68}(4) (2004), 117-150.


\bibitem[HRS]{HRS} D. Happel, I. Reiten, S.O. Smal\o, \emph{Tilting in abelian categories and quasitilted algebras},
 Mem. Amer. Math.
Soc. \textbf{120} (1996).

\bibitem[H]{H} R. Hartshorne, \emph{Algebraic Geometry}, 6th edition.  Springer-Verlag (1993). 


\bibitem[HKM]{HKM} M. Hoshino, Y. Kato, J-I. Miyachi, \emph{On t-structures and torsion theories induced by compact objects}, Journal of Pure and Applied Algebra, \textbf{167} (2002), 15-35.

\bibitem[J]{J} A. Jerem\'ias, \emph{El grupo fundamental relativo. Teor\'ia de Galois y localizaci\'on.}. Ph Thesis. Alxebra \textbf{56}. Universidade de Santiago de Compostela (1991). 



\bibitem[KN]{KN} B. Keller, P. Nicol\'as, \emph{Weight structures and simple dg modules for positive dg algebras}. Int. Math. Res. Notes \textbf{5} (2013), 1028-1078.

\bibitem[Ku]{Ku} E. Kunz, \emph{Introduction to Commutative Algebra and Algebraic Geometry}. Birkh$\ddot{a}$user (1985).

\bibitem[L]{L} D. Lazard, \emph{Autour de la platitude}. Bull. Soc. Math. France \textbf{97} (1969), 82-128.

\bibitem[MT]{MT} F. Mantese, A. Tonolo, \emph{On the heart associated with a torsion pair}, Topology and its Applications,
\textbf{159} (2012), 2483-2489.

\bibitem[M]{M} H. Matsumura, \emph{Commutative ring theory}.
Cambridge Stud. Adv. Math. \textbf{8}. Cambridge Univ. Press, 5th
edition (1997).



\bibitem[NS]{NS} P. Nicol\'as, M. Saor\'in, \emph{Parametrizing recollement data for triangulated categories}, J. Algebra \textbf{322} (2009), 1220-1250.


\bibitem[P]{P} B. Pareigis, \emph{Categories and functors}, Academic Press (1970).

\bibitem[PS]{PS} C.E. Parra, M. Saor\'in, \emph{Direct limits in the heart of a t-structure: the case of a torsion
pair}. J. Pure and Appl. Algebra \textbf{219}(9)  (2015), 4117-4143.

\bibitem[PS2]{PS2} C.E. Parra, M. Saor\'in, \emph{On hearts which are module categories}. To appear in J. Math. Soc. Japan. Preprint available at arXiv.org/abs/1403.1728


\bibitem[Po]{Po} N. Popescu, \emph{Abelian categories with applications to rings and modules}, London Math. Soc. Monogr \textbf{3}, Academic Press (1973).


\bibitem[ST]{ST} M.J. Souto, S. Trepode, \emph{t-structures on the bounded derived category of the Kronecker algebra}. Appl. Categ. Struct. \textbf{20}(5), 513-529.

\bibitem[St]{St} D. Stanley, \emph{Invariants of t-structures and classification of nullity classes}. Adv. Maths. \textbf{224} (2010), 2662-2689.

\bibitem[S]{S} B. Stenstr\"om, \emph{Rings of quotients},
Grundlehren der math.
Wissensch., \textbf{217},
Springer-Verlag, (1975).

\bibitem[V]{V} A. Verschoren, \emph{Relative invariants of sheaves}. Lect. Notes Pure Appl. Maths. \textbf{104}. Marcel Dekker, New York (1994).

\end{thebibliography}
\end{document}